\documentclass[11pt]{amsart}
\usepackage{amssymb,amsmath,amsthm,amsrefs}
\usepackage[mathscr]{eucal}
\usepackage{fullpage}
\usepackage{enumerate}
\usepackage{verbatim}
\usepackage{graphicx,float}

\usepackage{color}

\usepackage{psfrag}

\newtheorem{theorem}[equation]{Theorem}
\newtheorem{lemma}[equation]{Lemma}
\newtheorem{prop}[equation]{Proposition}

\newtheorem{definition}[equation]{Definition}
\newtheorem{remark}[equation]{Remark}
\newtheorem{assumption}[equation]{Assumption}

\numberwithin{equation}{section}

\newcommand{\R}{\mathbb{R}}
\newcommand{\C}{\mathbb{C}}
\newcommand{\Z}{\mathbb{Z}}

\newcommand{\Sph}{\mathbb{S}}
\newcommand{\T}{\mathbb{T}}

\newcommand{\geuc}{g_{_E}}
\newcommand{\gsph}{g_{_S}}

\newcommand{\kC}{k_{_C}}
\newcommand{\SigmaC}{\Sigma_{_C}}
\newcommand{\mC}{m_{_C}}

\renewcommand{\Re}{\operatorname{Re}}
\renewcommand{\Im}{\operatorname{Im}}

\newcommand{\stab}{\operatorname{Stab}}
\newcommand{\Grp}{\mathscr{G}} 
\newcommand{\Grefl}{\mathscr{G}_{refl}}
\newcommand{\Gsym}{\mathscr{G}_{sym}}
\newcommand{\Hgrp}{\mathscr{H}} 
\newcommand{\Kgrphat}{\widehat{\mathscr{K}}}

\newcommand{\vhat}{\widehat{V}_k}
\newcommand{\hhat}{\widehat{H}}

\newcommand{\Hcal}{\mathcal{H}}
\newcommand{\Lcal}{\mathcal{L}}
\newcommand{\Rcal}{\mathcal{R}}
\newcommand{\Qcal}{\mathcal{Q}}

\newcommand{\Sk}{\mathcal{S}_k}
\newcommand{\Stildekm}{\widetilde{\mathcal{S}}_{k,m}}

\newcommand{\SkC}{\mathcal{S}_{\kC}}

\newcommand{\Ck}{\mathcal{C}_k}
\newcommand{\Ckm}{\mathcal{C}_{k,m}}
\newcommand{\Ckmp}{\mathcal{C}'_{k,m}}
\newcommand{\Cmin}{\mathcal{C}_{min}}

\newcommand{\Cunder}{{\underline{\mathcal{C}}}}
\newcommand{\What}{\widehat{\mathcal{W}}}
\newcommand{\Wcal}{\mathcal{W}}

\newcommand{\xkm}{\widehat{X}_{k,m}}
\newcommand{\xkCm}{\widehat{X}_{\kC,m}}

\newcommand{\xC}{X_{_C}}

\newcommand{\phiC}{\phi_{_C}}

\newcommand{\intcirc}{\mathcal{C}}
\newcommand{\comptor}{\mathcal{T}}

\newcommand{\xx}{\ensuremath{\mathrm{x}}}
\newcommand{\yy}{\ensuremath{\mathrm{y}}}
\newcommand{\zz}{\ensuremath{\mathrm{z}}}
\newcommand{\rr}{\ensuremath{\mathrm{r}}}
\newcommand{\sss}{\ensuremath{\mathrm{s}}}
\newcommand{\ttt}{\ensuremath{\mathrm{t}}}

\newcommand{\ghat}{{\widehat{g}}}
\newcommand{\Ahat}{{\widehat{A}}}
\newcommand{\Lhat}{{\widehat{L}}}

\newcommand{\Grphat}{{\widehat{\Grp}}}

\newcommand{\Rhat}{{\widehat{\Rcal}}}
\newcommand{\Hhat}{{\widehat{\mathscr{H}}}}

\newcommand{\Rtowkn}{{\Rhat_{\Sk,n}}}
\newcommand{\RtowkmC}{{\Rhat_{{\Sk}_{_C},\frac{\mC}{km}}}}
\newcommand{\Rtor}{{\Rhat^{tor}}}

\newcommand{\rot}{\mathsf{R}}
\newcommand{\rote}{\widehat{\mathsf{R}}}
\newcommand{\transe}{\widehat{\mathsf{T}}}
\newcommand{\refle}{\widehat{\underline{\mathsf{R}}}}

\newcommand{\abs}[1]{\left\lvert#1\right\rvert}
\newcommand{\norm}[1]{\left\|#1\right\|}
\newcommand{\cutoff}[2]{\psi\left[ #1,#2 \right]}

\title{Minimal surfaces in the 3-sphere by desingularizing intersecting Clifford tori}

\author{Nikolaos Kapouleas}
\address{Department of Mathematics, Brown University, Providence, RI 02912}
\email{nicos@math.brown.edu}

\author{David Wiygul}
\address{Department of Mathematics, University of California, Irvine, CA 92697}
\email{dwiygul@uci.edu}

\begin{document}

\maketitle

\begin{abstract}
For each integer $k \geq 2$, we apply gluing methods to construct sequences
of minimal surfaces embedded in the round $3$-sphere. We produce two types of sequences,
all desingularizing collections of intersecting Clifford tori.
Sequences of the first type converge 
to a collection of $k$ Clifford tori intersecting with maximal symmetry along these two circles. 
Near each of the circles, after rescaling, the sequences converge smoothly on compact subsets
to a Karcher-Scherk tower of order $k$.
Sequences of the second type desingularize a collection of the same $k$ Clifford tori 
supplemented by an additional Clifford torus equidistant from the original two circles of intersection, 
so that the latter torus orthogonally intersects each of the former $k$ tori along a pair of 
disjoint orthogonal circles, 
near which the corresponding rescaled sequences converge to a singly periodic Scherk surface.
The simpler examples of the first type resemble surfaces constructed by Choe and Soret \cite{CS} 
by different methods where the number of handles desingularizing each circle is the same. 
There is a plethora of new examples which are more complicated and on which the number of handles for the two circles differs. 
Examples of the second type are new as well. 
\end{abstract}

\section{Introduction}

\subsection*{The general framework}
$\phantom{ab}$
\nopagebreak

After totally geodesic $2$-spheres, Clifford tori represent the next simplest minimal 
embeddings of closed surfaces in the round unit $3$-sphere $\Sph^3$. In fact Marques and 
Neves \cite{MN}, 
in their proof of the Willmore conjecture, have identified Clifford 
tori as the unique area minimizers among all embedded closed minimal surfaces of genus 
at least one, 
and Brendle \cite{Brendle} 
has shown they are the only embedded minimal tori, 
affirming a conjecture of Lawson. 
The first examples of higher-genus minimal surfaces in $\Sph^3$ were produced by Lawson himself 
\cite{Lawson}, and further examples were found later by Karcher, Pinkall, 
and Sterling \cite{KPS}. Both constructions proceed by solving Plateau's problem 
for suitably chosen boundary and extending the solution to a closed surface by reflection.

In this article we carry out certain constructions 
by using gluing techniques by singular perturbation methods.  
One begins with a collection of known embedded minimal surfaces. 
These ingredients are then glued together to produce a new embedded 
surface, called the initial surface, having small but nonvanishing 
mean curvature introduced in gluing. 
The construction is successful when the initial surfaces are close to a singular limit. 
The construction is then completed by perturbing the surface to minimality without sacrificing embeddedness.
Of course the size of the mean curvature and the feasibility of perturbing 
the surface so as to eliminate it both depend crucially on the design of the initial surface. 

Gluing methods have been applied extensively
and with great success in gauge theories by Donaldson, Taubes, and others.
In many geometric problems similar to the one studied in this article obstructions appear to solving the linearized equation.  
An extensive methodology has been developed to deal with this difficulty in a large class of geometric problems,  
starting with \cite{schoen}, \cite{kapouleas:annals}
and with further refinements in \cite{KapWe}.  
We refer to \cite{KapClay} for a general discussion of this gluing methodology and \cite{KapSchoen} 
for a detailed general discussion of doubling and desingularization constructions for minimal surfaces. 
In this article, however, 
we limit ourselves to constructions of unusually high symmetry 
(except in section \ref{further}) and this way we avoid these difficulties entirely. 

The first desingularization construction by gluing methods for minimal surfaces 
with intersection curves which are not straight lines was carried out 
in \cite{KapJDG} and serves as a prototype (see for example \cite{Nguyen,KKM,KapLi})  
for desingularizations of rotationally invariant surfaces 
with transverse intersections without triple or higher points. 
(An independent construction by Traizet \cite{Tr} has straight lines as intersections.) 
For one earlier application of the gluing methodology in the context of minimal surfaces in $\Sph^3$ see \cite{KY}, 
where a ``doubling'' construction of the Clifford torus is carried out; 
this work has been extended in \cite{Wiy} 
for ``stackings'' of the Clifford torus and in \cite{KapI} for doublings of the equatorial two-sphere. 
The present construction also glues tori, 
but by desingularization rather than doubling. 

The idea of a desingularization construction for intersecting minimal surfaces in a Riemannian three-manifold 
is to start with a collection of minimal surfaces intersecting along some curves 
and to produce a single embedded minimal surface by desingularizing the curves of intersection. 
Assuming transverse intersection, this is accomplished, at the level of the initial surface, 
through the replacement of a tubular neighborhood of each component curve of the intersection set by a surface 
which on small scales approximates a minimal surface in Euclidean space desingularizing the intersection of a collection of planes along a single line. 
In prior desingularization constructions the appropriate models for these desingularizing surfaces 
were furnished by the classical Scherk towers of \cite{Scherk}, 
which desingularize the intersection of two planes. 
One novelty of the present article is our use of the more general Karcher-Scherk towers, 
introduced in \cite{Kar}, which come in families desingularizing any number of intersecting planes and so accommodate curves of intersection whose complements, 
in small neighborhoods of the curves, contain more than four components.
Note that although having more than two minimal surfaces intersect along a curve is not a generic situation, 
it can happen in rotationally invariant cases as for example in the case of coaxial catenoids.  
Extending the results of \cite{KapJDG} to such situations for example is an interesting but difficult problem 
because one would have to use the full family of Karcher-Scherk towers as studied in \cite{perez2007classification}. 

A motivation for our construction is the observation \cite[Section 2.7]{Kar} 
that Lawson's surfaces may be regarded as desingularizations of a collection 
of great $2$-spheres intersecting with maximal symmetry along a common equator. 
In this article we pursue analogous constructions with tori instead of spheres 
as proposed in \cite[Section 4, page 300]{KapSchoen}. 
Pitts and Rubinstein have described one class of surfaces (item 10 on Table 1 of \cite{PR}), similar to some of our surfaces,  
to be obtained by min-max methods. 
Recently Choe and Soret \cite{CS} have produced examples by solving Plateau's problem for a suitably selected boundary, 
in the spirit of \cite{Lawson}. 
Their examples resemble the simpler examples we construct.  
(To prove they are the same one would have to prove that the solution of the Plateau problem is unique; see remark \ref{r:unique}). 
Our construction has been developed independently and is more general in ways we describe below,  
and our strategy is quite different, 
based as we already mentioned on gluing techniques by singular perturbation methods. 
On the other hand our methods work only for high-genus surfaces. 

To outline, 
we construct two infinite families of embedded minimal surfaces in $\Sph^3$. 
The first family consists of desingularizations of a configuration 
$\Wcal_k$ (see \ref{E:Wk}) 
of $k\ge2$ Clifford tori 
intersecting symmetrically along a pair of disjoint great circles $C_1$ and $C_2$ 
which lie on two orthogonal two-planes in $\R^4$. 
The second family consists of desingularizations of a configuration 
$\Wcal'_k$ (see \ref{E:Wk}) 
which is the previous one augmented by the Clifford torus which is equidistant from $C_1$ and $C_2$. 
In both cases the construction is based on choosing ``scaffoldings'', 
that is unions of great circles contained in the given configuration, 
and which we demand to be contained in the minimal surfaces we construct. 
Moreover, 
reflections with respect to the great circles contained in the scaffoldings are 
required to be symmetries of our constructions. 
We denote the scaffoldings we choose by 
$\mathcal{C}_{k,m}\subset\Wcal_k$   
or 
$\mathcal{C}'_{k,m}\subset\Wcal'_k$  
(see \ref{scaff}).

To construct the initial surfaces we replace tubular neighborhoods of the
intersection circles with surfaces modeled on appropriately scaled and
truncated maximally symmetric Karcher-Scherk towers. Towers with $2k$
ends are used along $C_1$ and $C_2$, while classical Scherk towers with $4$
ends are used along other circles of intersection 
(present only in 
$\Wcal'_k$). 
The replacements are made so that each initial surface
is closed and embedded, contains the applicable scaffolding, and is
invariant under reflection through every scaffold circle 
(see Definition \ref{initsdef}). 
These initial surfaces are perturbed then to minimality in a way which respects the reflections,  
so the surfaces produced are closed embedded minimal and still contain the scaffolding 
(see the Main Theorem \ref{mainthm}). 
Note that Lawson's approach also makes use of a scaffolding. 
Our approach, however, 
gives much more freedom in the number of handles we include in the fundamental domain, 
while in Lawson's method the fundamental domain is a disc so that Plateau's problem can be solved. 
This makes no difference when considering the original construction of Lawson in \cite{Lawson}: 
we expect that the construction with more handles in the fundamental domain will still produce a Lawson surface 
even though it does not a priori impose all the symmetries of the surface. 
When there are more than two circles of intersection involved, 
however, we can choose different numbers of 
handles on each of them and this gives a plethora of new surfaces as in the present constructions 
(see the Main Theorem \ref{mainthm}). 
It seems also rather daunting to try to construct even the simplest of our surfaces desingularizing 
$\Wcal'_k$ by Lawson's method. 

The present constructions motivate two important new directions for further study. 
First, what other similar desingularization constructions can be carried out in cases 
where there are obstructions due to less symmetry? 
One has to deal then with the obstructions in the usual way by introducing smooth families of initial surfaces 
with the parameters corresponding to the obstructions as in earlier work 
(see \cite{KapJDG,KapClay,KapSchoen,Kap:compact}). 
We discuss this question in Section \ref{further} and we provide some partial answers. 
Second, as remarked in the end of \cite[Section 4.2]{KapSchoen}, there are  
various natural questions about rigidity and uniqueness for the surfaces presently constructed which are similar to those asked 
\cite[questions 4.3, 4.4 and 4.5]{KapSchoen} 
about the Lawson surfaces. 
In particular we are currently working to prove that the surfaces desingularizing 
$\Wcal_k$ can not be smoothly deformed to surfaces desingularizing $k$ Clifford tori still 
intersecting along $C_1$ and $C_2$ but with different angles (that is they cannot ``flap their wings''). 
More precisely we hope to prove that even with reduced symmetries imposed so ``flapping the wings'' is allowed, 
there are no new Jacobi fields on our surfaces.

\subsection*{Outline of the approach}
$\phantom{ab}$
\nopagebreak

As we have already mentioned the main difficulty of this construction as compared to earlier 
results is proving that under the symmetries imposed there is no kernel on the Karcher-Scherk towers. 
As for the classical singly periodic Scherk surfaces ($k=2$) \cite{KapJDG}, 
this is achieved by subdividing the surface suitably and applying the Montiel-Ros approach \cite{MR}. 
Our approach is also somewhat different than usual in some other aspects and we employ the high symmetry we have 
available. 

\subsection*{Organization of the presentation}
$\phantom{ab}$
\nopagebreak

In Section 2 we study in sufficient detail the maximally symmetric Karcher-Scherk towers using the 
Enneper-Weirstrass representation and following \cite{Kar}. 
In Section 3 we study the geometry of the Clifford tori and the initial configurations we will 
be using later, their symmetries, and the symmetries and scaffoldings we will impose in our constructions later. 
In Section 4 we discuss in detail the construction of the initial surfaces and we study their geometry. 
In Section 5 we provide estimates for the geometric quantities on the initial surfaces. 
In Section 6 we study the linearized equation and estimate its solutions on the initial surfaces.  
We finally combine these results to establish the main theorem in Section 7. 
Finally in Section 8 we discuss further results using more technology.

\subsection*{Notation and conventions}
Given an open set $\Omega$ of a submanifold immersed in an ambient manifold endowed with metric $g$, an exponent $\alpha \in [0,1)$, and a tensor field $T$ on $\Omega$, possibly taking values in the normal bundle, we define the pointwise H\"{o}lder seminorm
\begin{equation}
  [T]_\alpha(x) = \sup_{y \in B_x} 
  \frac{\abs{T(x)-\tau_{yx}T(y)}_g}{d(x,y)^\alpha},
\end{equation}
where $B_x$ denotes the open geodesic ball, with respect to $g$, 
with center $x$ and radius the minimum of $1$ and the injectivity radius at $x$;
$\abs{\cdot}_g$ denotes the pointwise norm induced by $g$; $\tau_{yx}$ denotes parallel transport, also induced by $g$, 
from $y$ to $x$ along the unique geodesic in $B_x$ joining $y$ and $x$; and $d(x,y)$ denotes the distance between $x$ and $y$.

Given further a continuous positive function $f: \Omega \to \R$ and a nonnegative integer $\ell$, assuming that $T$ is a section of the bundle $E$ over $\Omega$ all of whose order-$\ell$ partial derivatives (with respect to any coordinate system) exist and are continuous, we set
\begin{equation}
  \norm{T: C^{\ell,\alpha}(E,g,f)} = \norm{T}_{\ell,\alpha,f} 
  = \sum_{j=0}^{\ell} \sup_{x \in \Omega} \frac{\abs{D^jT(x)}}{f(x)}
  + \frac{\left[D^\ell T\right]_\alpha(x)}{f(x)},
\end{equation}
where $D$ denotes the Levi-Civita connection determined by $g$.
In case
$E$ is the trivial bundle $\Omega \times \R$, 
instead of $C^{\ell,\alpha}(E,g,f)$ we write $C^{\ell,\alpha}(\Omega,g,f)$.
When $f\equiv 1$, we write just
$C^{\ell,\alpha}(\Omega,g)$, and when $\alpha=0$, we write just $C^\ell(\Omega,g)$.

Additionally, if $\Grp$ is a group acting a on a set $B$ 
and if $A$ is a subset of $B$,
then we refer to the subgroup
  \begin{equation}
  \label{stab}
    \stab_{\Grp}(A):=\{ \mathbf{g} \in \Grp \; | \; \mathbf{g}A \subseteq A \}
  \end{equation}
as the \emph{stabilizer} of $A$ in $\Grp$.
When $A$ is a subset of Euclidean $3$-space (or the round $3$-sphere), we will set
  \begin{equation}
  \label{Gsym}
    \Gsym(A):=\stab_{{Isom}} A,
  \end{equation}
where $Isom=O(3)$ (or $O(4)$).
For a subset $\mathcal{C}$ of Euclidean space (or the round $3$-sphere)
we define
  \begin{equation}
  \label{Grefl}
    \Grefl(\mathcal{C})
  \end{equation}
to be the group generated by reflections with respect to the lines
(or great circles) contained in $\mathcal{C}$.
Here reflection through a great circle can be defined as the restriction
to the $3$-sphere of reflection in $\R^4$ through the $2$-plane containing the circle.

If $\Grp$ is a group of isometries of a Riemannian manifold with a two-sided
immersed submanifold $\Sigma$, 
then we call $\mathbf{g} \in \stab_{\Grp}(\Sigma)$ \emph{even}
if $\mathbf{g}$ preserves the sides of $\Sigma$ 
and \emph{odd} if it exchanges them.
In this two-sided case, sections of the normal bundle
of $\Sigma$ may be represented by functions, on which $\stab_{\Grp} \Sigma$
then acts according to 
$(\mathbf{g}u)(x)=(-1)^{\mathbf{g}}u\left(\mathbf{g}^{-1}x\right)$,
where $(-1)^{\mathbf{g}}$ is defined to be $1$ for $\mathbf{g}$ even
and $-1$ for $\mathbf{g}$ odd.
We append a subscript $\Grp$ to the spaces of functions defined above
to designate the subspace which is
$\stab_{\Grp}(\Sigma)$-equivariant in this sense 
so that for instance
  \begin{equation}
  \label{invfun}
    C^{k,\beta}_{\Grp}(\Sigma,g,f) = \left.
    \left\{u \in C^{k,\beta}(\Sigma,g,f) \; \right| \; 
    u\left(\mathbf{g}^{-1}(x)\right) = 
    (-1)^{\mathbf{g}}u(x) \;
    \forall x \in \Sigma \;\; \forall \mathbf{g} \in \stab_{\Grp}(\Sigma) \right\}.
  \end{equation}

Finally, we make extensive use of cutoff functions,
and for this reason we fix a smooth function $\Psi:\R\to[0,1]$ with the following properties:
\begin{enumerate}[(i)]
\item $\Psi$ is nondecreasing,
\item $\Psi\equiv1$ on $[1,\infty]$ and $\Psi\equiv0$ on $(-\infty,-1]$, and
\item $\Psi-\frac12$ is an odd function.
\end{enumerate}
Given then $a,b\in \R$ with $a\ne b$,
we define a smooth function $\psi[a,b]:\R\to[0,1]$
by
\begin{equation}
\label{Epsiab}
\psi[a,b]=\Psi\circ L_{a,b},
\end{equation}
where $L_{a,b}:\R\to\R$ is the linear function defined by the requirements $L(a)=-3$ and $L(b)=3$.

Clearly then $\psi[a,b]$ has the following properties:
\begin{enumerate}[(i)]
\item $\psi[a,b]$ is weakly monotone,
\item $\psi[a,b]=1$ on a neighborhood of $b$ and $\psi[a,b]=0$ on a neighborhood of $a$, and
\item $\psi[a,b]+\psi[b,a]=1$ on $\R$.
\end{enumerate}

\subsection*{Acknowledgments}

The authors would like to thank Richard Schoen for his continuous support and interest in the results of this article. 
N.K. was partially supported by NSF grants DMS-1105371 and DMS-1405537.

\section{The Karcher-Scherk towers}
In \cite{Kar}
Karcher introduced generalizations of the classical singly periodic Scherk surfaces,
including the maximally symmetric Karcher-Scherk towers of order $k \geq 2$, 
which we will denote by $\Sk$. 
$\Sk$ is a singly periodic, complete minimal surface embedded in Euclidean $3$-space, asymptotic to $k$ planes intersecting at equal angles along a line.
This line, which we call the axis of $\Sk$, is parallel to the direction of periodicity.
The classical singly periodic Scherk tower \cite{Scherk} asymptotic to two orthogonal planes is recovered by taking $k=2$.
Although in this article we will only use 
$\Sk$ in our constructions, 
it is worth mentioning that there is a continuous family 
of singly periodic minimal surfaces with Scherk-type ends 
which has been studied by P{\'e}rez and Traizet in \cite{perez2007classification} 
and in which family $\Sk$ is the most symmetric member.

We proceed to outline the construction of Karcher \cite{Kar}.
Note that $\Sk$, which we will define in detail later, 
differs by a scaling and rotation from the surface described now.
Karcher considered the Enneper-Weierstrass data of
\begin{equation}
\label{WeiData}
\text{Gauss map } \zeta^{k-1} \text{ and height differential } \frac{1}{\zeta^k + \zeta^{-k}}\frac{d\zeta}{\zeta}
\end{equation}
on the closed unit disc in $\C$ punctured at the $2k$\textsuperscript{th} roots of $-1$.
The embedding defined by the data maps the origin to a saddle point, the $k$ line segments joining opposite roots of $-1$ to horizontal lines of symmetry intersecting at equal angles, the $2k$ radii terminating at roots of unity to alternately ascending and descending curves of finite length lying in $k$ vertical planes of symmetry, and the $2k$ circumferential arcs between consecutive roots of $-1$ to curves of infinite length lying alternately in one of two horizontal planes of symmetry.
The union of this region with its reflection through either horizontal plane of symmetry
is a fundamental domain
for the tower under periodic vertical translation.
The images of small neighborhoods of the roots of $-1$ are asymptotic to vertical half-planes,
which bisect the vertical planes of symmetry.

For future reference
the following proposition fills in the details of the above outline and summarizes the basic geometric
properties of $\Sk$,
including its symmetries and asymptotic behavior.
To state the lemma we make a few preliminary definitions.
First we define the sets $\hhat$, a union of horizontal planes,
and $\vhat$, a union of vertical planes, by
  \begin{equation}
    \hhat:=\bigcup_{n \in \Z} \{\zz=n\pi\}
    \quad \mbox{and} \quad
    \vhat:=\bigcup_{j \in \Z} \{ \yy=\xx \tan j\pi/k\},
  \end{equation}
whose intersection is the union of horizontal lines
  \begin{equation}
  \label{Ck}
    \Ck := \hhat \cap \vhat.
  \end{equation}
We enumerate the connected components of the complement of $\hhat \cup \vhat$ by
  \begin{equation}
      B_{k,l,j}:= \left\{ (\rr \cos \theta, \rr \sin \theta, \zz) \; : \;
        \rr \geq 0, \,
        \theta \in \left(\frac{j}{k}\pi, \frac{j+1}{k}\pi\right), \,
        \zz \in (l\pi,(l+1)\pi) \right\} 
  \end{equation}
for each $l,j \in \Z$,
and we also define a partition of
$\R^3 \backslash \left( \hhat \cup \vhat \right)$ into disjoint sets $A_k$ and $A_k'$
given by
  \begin{equation}
  \label{Ak}
      A_k:=\bigcup_{l+j \in 2\Z} B_{k,l,j} \quad \mbox{ and } \quad
      A_k':=\bigcup_{l+j \in 2\Z+1} B_{k,l,j}.
  \end{equation}

To describe the symmetries of $\Sk$ we write $\refle_P$ for reflection in $\R^3$ through the plane $P$,
$\transe_\ell^a$ for translation in $\R^3$ by $a$ units along the directed line $\ell$,
and $\rote_\ell^\phi$ for rotation in $\R^3$ through angle $\phi$
about the directed line $\ell$ (according to the usual orientation conventions).
A horizontal bar over a subset of $\R^3$ denotes its topological closure in $\R^3$,
and angled brackets enclosing a list of elements (or sets of elements) 
of $O(3)$ indicate the subgroup generated by all the elements listed or included in the sets mentioned.

\begin{prop}
\label{tower}
For every integer $k \geq 2$ there is a complete embedded minimal surface $\Sk \subset \R^3$ such that

\begin{enumerate}[(i)]

  \item $\Sk \cap \hhat = \Ck$, $\Sk \backslash \Ck \subset A_k$,
    and every straight line on $\Sk$ is contained in $\Ck$;

  \item for any connected component $B$ of $A_k$
    the intersection $\Sk \cap B$ is an open disc with 
    $\partial \left(\Sk \cap B \right)= \Ck \cap \overline{B}$ (the union of four horizontal rays);

  \item for each non negative integer $m$ the quotient surface
    $\Sk / \left\langle \transe_{\zz \text{-axis}}^{2m\pi} \right\rangle$
    has $2k$ ends and genus $(k-1)(m-1)$;

  \item $\Grefl(\Ck) \subsetneq \Gsym(\Sk) = \Gsym(A_k) = \Gsym(A_k') \subsetneq \Gsym(\Ck)$ 
    (recall \ref{Gsym} and \ref{Grefl})
    and
    $\Gsym(\Sk)$ acts transitively on the set of connected components of $\Sk \backslash \Ck$,
    the set of connected components of $A_k$, and the set of connected components of $A_k'$;

  \item $\Grefl(\Ck)=
    \left\langle 
      \rote_{\zz \text{-axis}}^{2\pi/k}, \,
      \rote_{\xx \text{-axis}}^\pi, \, \transe_{\zz \text{-axis}}^{2\pi}
    \right\rangle$;

  \item $\Gsym(\Sk)
    = 
    \left\langle
      \rote_{\zz \text{-axis}}^{2\pi/k}, \,
      \rote_{\xx \text{-axis}}^\pi, \, \transe_{\zz \text{-axis}}^{2\pi}, \,
      \refle_{\theta=\pi/2k}, \, \refle_{\zz=\pi/2}
    \right\rangle
    = 
    \left\langle
      \Grefl(\Ck), \, 
      \refle_{\theta=\pi/2k}, \, \refle_{\zz=\pi/2}
    \right\rangle $;

  \item $\Gsym(\Ck)
    =
    \Gsym(\Sk) \; \cup \; \rote_{\zz \text{-axis}}^{\pi/k}\Gsym(\Sk)
    =
    \Gsym(\Sk) \; \cup \; \transe_{\zz \text{-axis}}^{\pi}\Gsym(\Sk)$; and

  \item there exists $R_k>0$
    so that
    $\Sk \backslash \{ \sqrt{\xx^2+\yy^2} < R_k \}$ 
    has $2k$ connected components,
    all isometric by the symmetries, exactly one of which lies in the region
    $$
    \{ (\rr \cos \theta, \rr \sin \theta, \zz) \; : \; 
      \rr>0, \, \abs{\theta} < \frac{\pi}{2k}, \, \zz \in \R \},
$$
    and the intersection of this component with $\{\xx \geq R_k\}$
    is the graph
    $$ 
\{(\xx,W_k(\xx,\zz),\zz): \xx \geq R_k, \, \zz \in \R\} 
    $$ 
    over the $\xx\zz$-plane of a function
    $W_k: [R_k,\infty) \times \R \to \R$, 
which decays exponentially in the sense that $\forall j,\ell \ge0$ we have 
    \begin{equation}
    \label{towerdecay}
\abs{\partial_\xx^j \partial_\zz^{\ell} W_k(\xx,\zz)} \leq C(k,j+\ell) \, e^{-\xx}.
    \end{equation} 

\end{enumerate}
\end{prop}

\begin{proof}
The usual Inner-Weierstrass recipe
(see \cite{KarT} for example or any standard reference for the classical theory of minimal surfaces)
defines from the data \ref{WeiData} a minimal surface in $\R^3$ parametrized
by the closed unit disc in $\C$ punctured at the
$2k\textsuperscript{Th}$ roots of $-1$.
Requiring the parametrization to take the origin in $\C$ to the origin in $\R^3$,
it takes the form 
\begin{equation}
\label{WeiMap}
\begin{aligned} 
&\xx(w)=\Re \int_0^w \frac{1-\zeta^{2k-2}}{1+\zeta^{2k}} \, d\zeta
      =\frac{1}{k} \sum_{j=1}^{2k} (-\Re \omega_j) \ln  |w-\omega_j|, \\
&\yy(w)=\Re \int_0^w i\frac{1+\zeta^{2k-2}}{1+\zeta^{2k}} \, d\zeta
      =\frac{1}{k} \sum_{j=1}^{2k} (\Im \omega_j) \ln |w-\omega_j|, \\
&\zz(w)=\Re \int_0^w \frac{2\zeta^{k-1}}{1+\zeta^{2k}} \, d\zeta
      =\frac{1}{k} \sum_{j=1}^{2k} (-1)^{j-1} \arg \frac{\omega_j-w}{\omega_j},
\end{aligned}
\end{equation}
where $\omega_j = e^{i\pi \frac{1+2j}{2k}}$
is the $j$\textsuperscript{Th} $2k$\textsuperscript{Th} root of $-1$
and 
$\arg$ is the imaginary part of the branch of the logarithmic function
cut along the ray from $0$ to $-\infty$
and taking the value $0$ at $1$.

The symmetries can be read from the data by the following standard argument.
Looking at the expression for the metric
\begin{equation}
D's^2=\left(|\zeta|^{k-1}+\frac{1}{|\zeta|^{k-1}}\right)^2 \left| \frac{1}{\zeta^k+\zeta^{-k}} \frac{d\zeta}{\zeta}  \right|^2
\end{equation} in terms of the Enneper-Weierstrass data \ref{WeiData},
one identifies as intrinsic geodesics both
circumferential arcs on the unit circle
and diametric segments joining opposite $2k\textsuperscript{th}$ roots of $\pm 1$.
Looking next at the expression for the second fundamental form
\begin{equation}
\label{tow2ff}
A_{tow}(V,V) = 2(k-1) \Re 
  \left[
  \left(\frac{d\zeta}{\zeta}V\right) \left(\frac{1}{\zeta^k+\zeta^{-k}}\frac{d\zeta}{\zeta}V \right)
  \right]
\end{equation}
in terms of the Enneper-Weierstrass data, it becomes apparent that the diametric segments joining opposite $2k$\textsuperscript{th} roots of $-1$ are extrinsic geodesics as well, which are therefore mapped to Euclidean lines, while the diametric segments joining opposite $2k$\textsuperscript{th} roots of unity along with the circumferential arcs are lines of curvature, which, being also intrinsic geodesics, are necessarily mapped to planar curves.

Consultation of \ref{WeiMap} confirms that
(a) the straight lines lie along the intersection of the single horizontal plane $\zz=0$
with the vertical planes of the form $\theta=\pi/2k + n\pi/k$ for $n \in \Z$, 
(b) the images of the circumferential arcs lie, alternatingly,
in the two horizontal planes $\zz=\pm \pi/2k$,
and (c) the images of the remaining lines of curvature lie, consecutively, in
the vertical planes of the form $\theta= n\pi/k$ for $n \in \Z$.

Below we will check that the parametrization \ref{WeiMap}
is in fact an embedding of the punctured unit disc.
In fact we will show that the image of the punctured sector
  \begin{equation}
    D:=\left\{re^{i\theta}: 0 \leq r \leq 1, \, 0 \leq \theta \leq \frac{\pi}{2k}\right\}
      \backslash \{\omega_1\}
  \end{equation}
is embedded and intersects the planes
$\zz=0$, $\zz=\pi/2k$, $\theta=0$, and $\theta=-\pi/2k$ only along
the curves just mentioned. 
The reflection principle for harmonic functions then implies
that this image may be extended to a
complete embedded minimal surface $\check{\mathcal{S}}_k$,
invariant under reflection through any line in $\frac{1}{k}\Ck$
and through any plane of the form $\theta=\pi/2k+j\pi/k$
or $\zz=\pi/2k+j\pi/k$ for $j \in \Z$.
We define $\Sk:=k \rote_{\zz \text{-axis}}^{\pi/2k} \check{\mathcal{S}}_k$.

Items (ii) and (iii) and the first two claims of item (i) follow immediately,
as do the containments $\Grefl(\Ck) \subsetneq \Gsym(A_k) \subseteq \Gsym(\Sk)$,
considering that item (v) and the equalities
$\Gsym(A_k)=\Gsym(A'_k)= \left\langle
      \rote_{\zz \text{-axis}}^{2\pi/k}, \,
      \rote_{\xx \text{-axis}}^\pi, \, \transe_{\zz \text{-axis}}^{2\pi}, \,
      \refle_{\theta=\pi/2k}, \, \refle_{\zz=\pi/2}
    \right\rangle$
are clear from the definitions (\ref{Ck} and \ref{Ak}) of $\Ck$, $A_k$, and $A'_k$ alone.
To see the containment
$\Gsym(\Sk) \subseteq \Gsym(\Ck)$
note that any symmetry of the tower must permute the asymptotic planes,
so must preserve their intersection,
so must take any line in $\Ck$ to a line orthogonally intersecting the $\zz$-axis;
we are assuming (and confirm with the maximum-principle argument below)
that $\Sk$ intersects the $\zz$-axis only where $\Ck$ does,
and from the expression (\ref{tow2ff}) for the second fundamental form
we know we have already accounted for all lines on $\Sk$ through such points.
The equalities
$\Gsym(\Ck)= \Gsym(A_k) \; \cup \; \rote_{\zz \text{-axis}}^{\pi/k} \Gsym(A_k)
     = \Gsym(A_k) \; \cup \; \transe_{\zz \text{-axis}}^{\pi} \Gsym(A_k)$
are also immediate consequences of the definitions.
In particular reflection through the $\zz=0$ plane preserves $\Ck$,
but it does not preserve $\Sk$
(because, for example,
the normal to $\Sk$ is not constantly vertical along the lines $\Sk \cap \{\zz=0\}$),
so $\Gsym(\Sk) \neq \Gsym(\Ck)$. 
Now, since $\Gsym(A_k)$ has index $2$ in $\Gsym(\Ck)$
and $\Gsym(A_k) \subseteq \Gsym(\Sk) \subsetneq \Gsym(\Ck)$,
in fact $\Gsym(A_k)=\Gsym(\Sk)$.

Thus we have checked items (ii)-(vii), as well as the first two claims of (i),
and the transitivity claim in (iv) is now obvious.
To verify the remaining claim in (i), note that any straight line in $\Sk$,
by virtue of the latter's minimality, is a line of reflectional symmetry. 
On the other hand reflection through a straight line in $\R^3$
preserves $A_k$ only if the line is the $\zz$-axis (and then only for $k$ even)
or if it is contained in $\Ck$ or
$\transe_{\zz \text{-axis}}^{\pi/2}\rote_{\zz \text{-axis}}^{\pi/2k}\Ck$.
It is easy to see, however, that of these lines only those contained in $\Ck$ lie on the surface.
(For example \ref{WeiData} and \ref{WeiMap} reveal that at easily identified points where
any of the other lines do intersect the surface the normal to the surface there is parallel to the line.)

In addition to checking (viii), it remains to show that the image of the region $D$ is embedded
and intersects the planes $\zz=0$, $\zz=\pi/2k$, $\theta=0$, and $\theta=-\pi/2k$
as described above.
That the tower is embedded may be established by recognizing the conjugate surface of the region between two consecutive horizontal planes of symmetry as a graph---specifically the solution to the Jenkins-Serrin problem \cite{JS} on a regular $2k$-gon---and then appealing to a theorem of Krust. See \cite{KarT} for details on this approach. Alternatively we show embeddedness more directly as follows
and in the process identify the intersection of the image of $D$ with these four planes.

Recall that $D$ is the punctured sector $\{re^{i\theta}: 0 \leq r \leq 1, \, 0 \leq \theta \leq \frac{\pi}{2k}\} \backslash \{\omega_1\}$. From \ref{WeiMap} it is clear that $d\zz$ is positive along the radial segment from the origin to $1$ and vanishes along both the radial segment from the origin to $\omega_1$ and the circular arc from $1$ to $\omega_1$. A sufficiently small circular arc, centered at $\omega_j$, which originates on the segment from $0$ to $\omega_1$ and terminates on the circumferential arc from $\omega_1$ to $1$, can be seen from \ref{WeiMap} to have height monotonically increasing from $0$ to $\frac{\pi}{2k}$. The maximum principle then implies that the image of $D$ is contained in the slab
$\{0 \leq \zz \leq \frac{\pi}{2k}\}$
and intersects $\zz=0$ only along the (straight, horizontal) image of the radial segment to $\omega_1$
and $\zz=\pi/2k$ only along the (horizontal) image of the circumferential arc.

Similarly, from \ref{WeiMap} one may readily check monotonicity of $d\xx$, $d\yy$, and $(\Re \omega_1) d\yy + (\Im \omega_1) d\xx$  on the boundary curves of $D$ in order to establish that the boundary has image contained in the wedge
$\{(\rr \cos \theta, \rr \sin \theta, \zz) \; : \;
  \rr \geq 0, \, -\frac{\pi}{2k} \leq \theta \leq 0, \, \zz \in \R \}$.
Moreover, \ref{WeiMap} reveals that in $D$
\begin{align}
& \lim_{w \to \omega_1} \xx(w) = \infty, \\
& \lim_{w \to \omega_1} \yy(w) = -\infty, \text{ and} \\
& \lim_{w \to \omega_1} [(\Re \omega_1)\yy(w) + (\Im \omega_1)\xx(w)] = 0,
\end{align}
so that another application of the maximum principle (to the harmonic coordinate functions $\xx$, $\yy$, and $(\Re \omega_1)\yy + (\Im \omega_1)\xx$, the last extended to the closure of $D$) establishes that the image of $D$ is contained in 
$\{(\rr \cos \theta, \rr \sin \theta, \zz) \; : \;
  \rr \geq 0, \, -\frac{\pi}{2k} \leq \theta \leq 0, \, 0 \leq \zz \leq \frac{\pi}{2k}\}$ 
and intersects $\theta=0$ only along the image of the radial segment to $1$
and $\theta=-\pi/2k$ only along the image of the radial segment to $\omega_1$.

Because of the symmetries it therefore suffices to show that the Enneper-Weierstrass parametrization restricts to $D$ as an embedding. To this end observe first that each level curve of $\yy$ in $D$ is connected. Indeed one can verify that $d\yy$ vanishes nowhere on 
$D$ and has norm (relative to $\abs{dz}^2$) tending to infinity at $\omega_1$,
so $\frac{\nabla \yy}{|\nabla \yy|^2}$ defines a smooth vector field on the closure of $D$. The corresponding flow for time $t$, when it exists, 
maps points (other than the fixed point $\omega_1$) 
with $\yy=\yy_0$ to points with $\yy=\yy_0+t$. 
Examining the field at the boundary, one can check that the backward flow of a point in $D$ exits $D$ only after it reaches the segment joining $0$ and $\omega_1$,
while the forward flow leaves $D$ only through the real boundary segment.
Thus, if the flow for time $t \leq 0$ of a point $x$ on the real segment from $0$ to $1$ lies in $D$, then the flow for time $t$ of any real point to the right of $x$ will also lie in $D$. Now given $w_1, w_2 \in D$ with $\yy(w_1)=\yy(w_2)=\yy_0<0$, the flow for time $-\yy_0$ takes 
each point to a point on the real segment of $D$ (since by the maximum principle and earlier monotonicity arguments this segment is the entirety of the $\yy=0$ curve in $D$). 
By the previous considerations the flow for time $\yy_0$ exists for every point on the real segment joining these points, so, 
the flow for time $\yy_0$ being a continuous function of the initial point, we get a level $\yy=\yy_0$ path joining $w_1$ and $w_2$.

Now suppose there exist points $w_1, w_2 \in D$ such that $\xx(w_1)=\xx(w_2)$ and $\yy(w_1)=\yy(w_2)$. Then there is a path contained in a single level curve of $\yy$ joining $w_1$ and $w_2$, and, if $w_1$ and $w_2$ are distinct, by the mean value theorem there exists a third point $w_3$ between $w_1$ and $w_2$ on this path at which $d\xx$ vanishes along the path. Thus at $w_3$ the gradients $\nabla \xx$ and $\nabla \yy$ must be parallel. One finds from \ref{WeiMap} that these gradients are parallel only on the circular boundary of $D$, where they are tangential to the boundary, and therefore they are endpoints of level curves, so that we may assume $w_3$ is not such a point. Thus we conclude that $w_1=w_2$, showing not only that the Enneper-Weierstrass parametrization is an embedding but also that its image of the unit disc is actually a graph over a region in the $\zz=0$ plane.

Now we prove (viii).
It is clear from \ref{WeiMap} that for $R$ sufficiently large
the set of $w$ in the unit disc with $\xx^2(w)+\yy^2(w) > R^2$
has $2k$ components,
each containing exactly one of the $2k$\textsuperscript{th} roots of $-1$.
We 
define
\begin{align}
\sss(w) &= \Re \omega_1 \xx(w) - \Im \omega_1 \yy(w) \text{ and} \\
\ttt(w) &= \Im \omega_1 \xx(w) + \Re \omega_1 \yy(w),
\end{align}
so that $\ttt(w)$ is
the signed distance of the image of $w$ from the plane $\theta=-\pi/2k$
and 
\begin{equation}
(\xx(w),\yy(w),\zz(w))
= \sss(w) (\Re \omega_1, -\Im \omega_1, 0)
+ \zz(w)(0,0,1)
+ \ttt(w)(\Im \omega_1, \Re \omega_1, 0),
\end{equation}
the three terms on the right being pairwise orthogonal.
We will show that for $R_k$ sufficiently large
there is a correspondingly small neighborhood
$\Omega_1$ of $\omega_1$ in the closed unit disc
such that the map
$f: \C \backslash \bigcup_{j=1}^{2k} \{ c\omega_j \; | \; c \in [1,\infty) \} \to \C$
defined by
  \begin{equation}
    f(w):=(\sss(w),\zz(w))
  \end{equation}
restricts to a diffeomorphism from $\Omega_1$ onto the half-strip
$[R_k,\infty) \times \left[-\frac{\pi}{2k},\frac{\pi}{2k}\right]$.

Working from \ref{WeiMap} we find
  \begin{equation}
    k\ttt(w)
    = 
    \ln \prod_{j=2}^k 
      \abs{ \frac{w-\omega_j}{w-\omega_1^2 
      \overline{\omega_j}}}^{\Im \overline{\omega_1}\omega_j},
  \end{equation}
  \begin{equation}
    \begin{aligned}
    k\zz(w)
    &= \arg \left(1-\frac{w}{\omega_1}\right)
      + \arg \prod_{j=2}^{2k} \left(1-\frac{w}{\omega_j}\right)^{(-1)^{j-1}} \\
    &= \arg \left(1-\frac{w}{\omega_1}\right) + \psi_\zz(w),
    \end{aligned}
  \end{equation}
and
\begin{equation}
  \begin{aligned}
    k\sss(w) &= -\ln |w-\omega_1| + \ln |w+\omega_1| 
      - \sum_{j=2}^{k} \Re \overline{\omega_1}\omega_j 
      \ln |w-\omega_j||w-\omega_1^2 \overline{\omega_j}| \\
    &= -\ln |w-\omega_1| + c + \psi_\sss(w),
  \end{aligned}
\end{equation}
where $c=\displaystyle{\lim_{w \to \omega_1}} \left(k\sss(w) + \ln \abs{w-\omega_j}\right)$
and $\psi_\sss$ and $\psi_\zz$ are defined by the equalities where they are introduced.
Then
$\displaystyle{\lim_{w \to \omega_1} \psi_\sss(w)=\lim_{w \to \omega_1} \psi_\zz(w)}=0$, and for each nonnegative integer $\ell$ there exists a constant 
$C(k,\ell)>0$ such that
\begin{equation}
\sup_{w \in \Omega_1} \sum_{j=0}^\ell \left( \abs{\partial_w^j \partial_{\overline{w}}^{\ell-j} \psi_\sss(w)} + \abs{ \partial_w^j \partial_{\overline{w}}^{\ell-j} \psi_\zz (w)} \right) < C(k,\ell).
\end{equation}

Defining the map $g: \C \to \C$ by
  \begin{equation}
    g(w):=\omega_1\left(1-e^{-k\overline{w}+c}\right),
  \end{equation}
we see that for $R$ sufficiently large the composite
$f \circ g\left|_{[R,\infty) \times (-\pi,\pi)}\right.$ is well-defined
(identifying $\C$ with $\R^2$ as usual)
and
  \begin{equation}
    f(g(w))=w+\frac{1}{k}(\psi_\sss+i\psi_\zz)(g(w)),
  \end{equation}
so,
since
$\displaystyle{\lim_{\Re w \to \infty} g(w)=\omega_1}$
and $\displaystyle{\lim_{\Re w \to \infty} \partial^j_{\overline{w}} g(w)=0}$
for each integer $j>0$,
by taking $R$ large enough we can ensure that
$f \circ g\left|_{[R,\infty) \times (-\pi,\pi)}\right.$
is a small perturbation of the identity
and so a diffeomorphism with image containing the half-strip
$[R_k,\infty) \times \left[-\frac{\pi}{2k},\frac{\pi}{2k}\right]$
for some $R_k>R$.

Thus $f$ itself restricts to a diffeomorphism from some region $\Omega_1$
onto this half-strip, as asserted above,
which shows that the image of $\Omega_1$
under the Enneper-Weierstrass parametrization
is the graph of $\ttt \circ f^{-1}$ over the half-strip
$\{(\rr \cos -\pi/2k, \rr \sin -\pi/2k, \zz) \; : \;
  \rr \geq R_k, \abs{\zz}\leq \pi/2k\}$.
Since $\ttt$ is smooth on $\Omega_1$, $\ttt(\omega_1)=0$,
and $f^{-1}=g \circ (f \circ g)^{-1}$,
we finally obtain the estimates
\begin{equation}
\abs{ \partial_\sss^j \partial_\zz^{\ell-j} (\ttt \circ f^{-1}) (\sss+i\zz)} \leq C(k,\ell)e^{-k\sss}.
\end{equation}
\end{proof}

The union $\Ck$ of all horizontal lines on $\Sk$ can be regarded as a scaffolding for the tower,
but we emphasize that a tower is not uniquely determined by a choice of scaffold: 
\begin{remark}
\label{towerremark} 
(i) The surface $\transe_{\zz \text{-axis}}^\pi \Sk = \rote_{\zz \text{-axis}}^{\pi/k}\Sk$
satisfies all conditions in the lemma provided the roles of $A_k$ and $A_k'$ are reversed in the statement,
so in particular its intersection with $\hhat$ is $\Ck$.

(ii) For $m$ a positive integer the surface $m^{-1}\Sk$ also has $\Ck$ as its intersection with $\hhat$,
but the quotient surface
$\left(m^{-1}\Sk\right) / \left\langle \transe_{\zz \text{-axis}}^{2\pi} \right\rangle$
has genus $(k-1)(m-1)\ne0$ instead of $0$ (recall \ref{tower}.iii). 
\end{remark}

\section{The initial configurations}
\label{initial}

\subsection*{The Clifford tori}

In this subsection we discuss the Clifford tori and their geometry. 
We first introduce some helpful notation. 
Given a great circle $C$ in $\Sph^3$ we will write $C^\perp$ for the furthest great circle from it. 
(Note that the points of $C^\perp$ are at distance $\pi/2$ in $\Sph^3$ from $C$ and any point of $\Sph^3\setminus C^\perp$ 
is at distance $<\pi/2$ from $C$). 
As viewed from $\R^4$, the planes containing $C$ and $C^\perp$ are orthogonal complements. 
On the other hand, $C$ and $C^\perp$ may be regarded as parallel in that the function on $\Sph^3$ measuring distance
from one of the circles is constant on the other.
(This relation of parallelism between two great circles in $\Sph^3$
is not transitive.) 
Another useful characterization of $C^\perp$ identifies it as the set
of poles of great two-spheres with equator $C$.

\begin{definition}[Clifford tori] 
\label{D:tori} 
If $C$ and $C^{\perp}$ are as above, 
then we call them {\em{totally orthogonal}.} 
We define the Clifford torus $\T[C]$ 
with ``axis-circles'' $C$ and $C^\perp$ 
to be the set of points in $\Sph^3$ equidistant from $C$ and $C^\perp$. 
\end{definition}

$\T[C]$ can be alternatively defined as the set of points which are at a distance ${\pi}/{4}$ from $C$, 
or equivalently at a distance ${\pi}/{4}$ from $C^\perp$. 
Clearly $\T[C]=\T[C^\perp]$.  
The set of points at distance $\pi/4$ from $\T[C]$ is $C\cup C^\perp$ 
and the set of points at distance $<\pi/4$ from $\T[C]$ is  
$\Sph^3\setminus (C\cup C^\perp)$.  
$\T[C]$ is a flat, square, embedded torus, foliated by the circles,
of radius $\frac{1}{\sqrt{2}}$ in $\R^4$, where great two-spheres having $C$
as equator and poles on $C^\perp$ intersect $\T$, and also by the circles,
orthogonally intersecting these, where great two-spheres having equator $C^\perp$
and poles on $C$ intersect $\T$ in pairs on opposite sides of the equator.

Evidently any element of $O(4)$ that preserves $C \cup C^\perp$ as a set is a symmetry of $\T[C]$. 
The group of these symmetries includes arbitrary rotation or reflection in the two circles 
as well as orthogonal transformations exchanging the circles. 
To proceed further we have the following. 

\begin{definition}[The rotations $\rot_C^\phi$]  
\label{D:rot} 
Given $C$ as above, $\phi \in \R$,
and assuming an orientation chosen on the totally orthogonal circle $C^\perp$,
we define $\rot_C^\phi$, 
rotation about $C$ by angle $\phi$, 
to be the element of $SO(4)$ preserving $C$ pointwise 
and rotating the totally orthogonal circle $C^\perp$ through an angle $\phi$, 
according to the chosen orientation on $C^\perp$. 
Just as well $\rot_C^\phi$ may be called rotation in $C^\perp$ by angle $\phi$.  
\end{definition} 

We assume now that orientations have been chosen for both $C$ and $C^\perp$. 
(Of course, after orienting $\Sph^3$, an orientation on a circle $C$ determines an orientation on $C^\perp$.)
We define then two $SO(2)$ subgroups of $O(4)$ by 
\begin{equation}
\label{E:Hgrp} 
\Hgrp^\pm_C:= 
\{\rot_C^\phi \rot_{C^\perp}^{\pm\phi}\,:\, \phi\in \R\},  
\end{equation}
each of whose elements rotates $C$ and $C^\perp$ simultaneously by a common angle, 
the two subgroups being distinguished by the relative sense of rotation in the circles.

It is easy to see that any such one-parameter subgroup of $O(4)$,
acting by common rotation in a pair of totally orthogonal circles, 
has only great circles as orbits.
Moreover by an easy calculation the orbits of 
$\Hgrp^+_C$ intersect the orbits of $\Hgrp^-_C$ at an angle equal to twice  
the distance from $C$. 
Consequently, $\T[C]$ itself is foliated by two such families of great circles, 
with the circles of one family intersecting the circles of the other orthogonally. 
More explicitly let $D,D'\subset\T[C]$ be great circles through some given point $p\in\T[C]$ 
with $D$ an orbit of $\Hgrp^+_C$ and $D'$ an orbit of $\Hgrp^-_C$. 
Then one family consists of the images of $D$ under the action of $\Hgrp^-_C$  
and the other of the images of $D'$ under the action of 
$\Hgrp^{+}_C$.  

Reflection through $D$ (that is $\rot_D^\pi$) preserves 
the great circles in $\T[C]$ which are orthogonal to $D$.  
Since these great circles foliate $\T[C]$ we conclude that the reflection 
$\rot_D^\pi$ is a symmetry of $\T[C]$. 
It follows that any point of a Clifford torus lies on a circle of reflection (two in fact)
and this immediately implies the minimality of the torus (since the symmetry allows
the mean curvature nowhere to point). 
Of course the two great circles through a point are asymptotic lines 
for the second fundamental form;
the short circles mentioned above---latitudes $\pi/4$ from the poles
of great two-spheres with equator one of the axis-circles---bisect
these and as such are circles of principal curvature.
These circles have curvature $\sqrt{2}$ in $\R^4$ and $1$ in $\Sph^3$,
showing that the second fundamental form of $\T[C]$ has constant norm $\sqrt{2}$.
This also implies that there are no great circles on $\T[C]$ other than the orbits of 
$\Hgrp^\pm_C$ as above, 
and therefore reflection through any great circle in $\T[C]$ 
is a symmetry of $\T[C]$. 

Given now $D$ and $D'$ as above,  
$\rot_C^{\phi}\rot_{C^\perp}^{-\phi}$ rotates the points of $D$ 
to a parallel great circle along orthogonal geodesic segments parallel to $D'$. 
Taking $\phi=\pi/2$ we conclude that 
$\rot_C^{\pi/2}\rot_{C^\perp}^{-\pi/2}D$, 
and similarly 
$\rot_C^{\pi/2}\rot_{C^\perp}^{\pi/2}D'$,  
are great circles which are parallel 
to $D$ and $D'$ respectively on $\T[C]$ 
at distance $\pi/2$. 
This implies that 
\begin{equation} 
\label{E:rotDD} 
\rot_C^{\pi/2}\rot_{C^\perp}^{-\pi/2}D=D^\perp 
\quad\text{ and } \quad 
\rot_C^{\pi/2}\rot_{C^\perp}^{\pi/2}D'=D'^\perp.
\end{equation} 
Moreover, the elements of $\Hgrp^+_C$ 
which were originally expressed as common rotations in $C$ and $C^\perp$,
are just as well common rotations in $D$ and $D^\perp$ 
(since $\rot_C^\phi\rot_{C^\perp}^\phi$ and $\rot_D^\phi\rot_{D^\perp}^\phi$
obviously have the same action on $D$ and $D^\perp$,
which are orbits of $\Hgrp_C^+$
and are the intersections with $\Sph^3$ of two $2$-planes spanning $\R^4$) 
and hence 
$\Hgrp^+_C = \Hgrp^+_D$ if we choose 
orientations on $D$ and $D^\perp$ appropriately
(or if we choose an orientation on $\Sph^3$).  
Similarly 
$\Hgrp^-_C = \Hgrp^+_{D'}$
(since $D'$ and $D'^\perp$ are orbits of $\Hgrp_C^-$).  

By the above $\T[C]$ is ruled by two families of great circles and any great circle on $\T[C]$ 
belongs to one of the two families.  
Any two circles in a single family are thereby not only parallel in $\T[C]$ 
but also parallel in $\Sph^3$ in the sense of distance as defined above.
Moreover $\T[C]$ is generated by twisting about any one of its great circles $D$
an orthogonally intersecting great circle $D'$, 
where the twisting is common rotation in $D$ and $D^\perp$:
traversing $D$, the conormal co-rotates with the position vector, tracing
out $D$ and $D^\perp$ at equal rates. In this sense the Clifford torus resembles the helicoid,
but the helicoid is just singly ruled, while all of the Clifford torus' great
circles are on equal footing. Additionally, while each helicoid is either right
or left-handed, every Clifford torus is ambidextrous, being right-handed along the
circles of one foliation and left-handed along the others.
More precisely, given an orientation on $\Sph^3$,
a Clifford torus $\T[C]$, and a great circle $D \subset \T[C]$,
we call $D$ right-handed if $\Hgrp_D^+\T[C]=\T[C]$
and left-handed if $\Hgrp_D^-\T[C]=\T[C]$. 

Furthermore, an element $\rot_D^\phi$ of $O(4)$
fixing a great circle $D$ on $\T[C]$ pointwise but rotating $D^\perp$
along itself 
will of course preserve $D$ and $D^\perp$ as sets but rotate $\T[C]$.
Such considerations reveal the existence of a one-parameter family
$\{\rot_D^\phi\T[C] \, : \, \phi \in \R\}$ 
of Clifford tori, 
each intersecting $\T[C]$ transversely along $D$ and $D^\perp$ at a constant angle $\phi$.
Since two great circles intersecting $D$ orthogonally can meet only
on $D$ or $D^\perp$ (unless they coincide everywhere), the ruling forbids distinct
members of this family from intersecting anywhere else.
More precisely for $\phi_1,\phi_2\in\R$, 
$\rot_D^{\phi_1}\T[C] = \rot_D^{\phi_2}\T[C]$ 
when $\phi_1=\phi_2 \pmod \pi$ and  
$\rot_D^{\phi_1}\T[C] \cap \rot_D^{\phi_2}\T[C]= D\cup D^\perp$ otherwise. 

There is a second one-parameter family of Clifford tori through $D$ and $D^\perp$,
having the opposite chirality along both, compared to the first family.
This family can be obtained from the first by, for example, reflection through a great two-sphere
containing $D$ (or $D^\perp$). 
A member of one family then intersects each member of the other family along two pairs of circles: 
$D$ and $D^\perp$ as well as two more circles, 
$D'$ and $D'^\perp$, 
where $D$ and $D^\perp$ intersect 
$D'$ and $D'^\perp$ 
orthogonally. 
The angle of intersection varies (from $0$ to $4\pi$) along these circles, 
with tangency of the surfaces occurring at the eight points where the four circles intersect in pairs.

Note that by the above 
if $D$ is any great circle, 
then 
any Clifford torus $\T$ which contains $D$ contains also $D^\perp$. 
Moreover $\T$ belongs to one of two families of Clifford tori containing $D\cup D^\perp$ as described above. 
Also for each point $p\in \Sph^3\setminus(D\cup D^\perp)$ there is a unique torus in each family which contains it, 
and on that torus there is a unique great circle through $p$ parallel to $D$ and $D^\perp$ in the torus. 
Additionally note that if $\T=\T[C]$,
then, fixing a point $q \in D$, there is a point $p \in C$
at distance $\pi/4$ from $q$, and in fact $p$ must lie a distance
$\pi/4$ from $D$, since otherwise there would be a point on $D$
less than $\pi/4$ from $C$, violating the definition of $\T[C]$.
On the other hand, $D$ is preserved by $\Hgrp_C^+$ or $\Hgrp_C^-$
while both subgroups act transitively on $C$,
so actually the distance from every point on $C$ to $D$ is $\pi/4$.
We conclude that whenever $D$ is a great circle on $\T[C]$,
$C$ is also a great circle on $\T[D]$. 

Now suppose two Clifford tori $\T[C]$ and $\T[C']$ have intersection $D \cup D^\perp$, where $D$ is a great circle.
Since $D \subset \T[C]$ and $D \subset \T[C']$, it follows from the last observation in the previous paragraph that
$C \cup C' \subset \T[D]$.
Moreover $\T[C]$ and $\T[C']$ must have the same chirality along $D$ (since otherwise their intersection would be larger as above), 
so one torus can be obtained from the other by a rotation about $D$, and therefore $C$ and $C'$ (and $C^\perp$ and $C'^\perp$) must be parallel great circles on $\T[D]$. 
Conversely, if $C$ and $C'$ are parallel great circles on $\T[D]$, 
then $C'$ can be obtained from $C$ by a rotation about $D$, 
so $\T[C']$ is obtained from $\T[C]$ by the same rotation, 
and therefore $\T[C]$ and $\T[C']$ intersect along just $D$ and $D^\perp$ as described above. 
Thus two Clifford tori intersect along a single pair of totally orthogonal great circles precisely when all four of their axis-circles are parallel 
on the Clifford torus equidistant from the intersection circles. 
Moreover if $\T[C]$ and $\T[C']$ orthogonally intersect along great circles $D$ and $D^\perp$, 
then each contains the axis-circles of the other: 
by definition of $\T[C]$, geodesics emanating from $\T$ orthogonally hit $C \cup C^\perp$ at distance $\pi/4$ from $\T$, 
but geodesics intersecting $\T[C]$ orthogonally at $D$ by assumption lie on $\T[C']$.
 
On the other hand two Clifford tori
$\T[C]$ and $\T[C']$ 
have intersection
$D\cup D^\perp\cup D'\cup {D'}^\perp$ precisely when
$C$ and $C'$ intersect orthogonally,
in which case $C^\perp$ and $C'^\perp$ also intersect orthogonally
and moreover $C \cup C^\perp \cup C' \cup C'^\perp=\T[D]\cap\T[D']$.
Indeed, suppose first that $\T[C] \cap \T[C']=D\cup D^\perp\cup D'\cup{D'}^\perp$.
Since $D$ lies on $T[C]$ and $T[C']$,
$C$ and $C'$ must be great circles on $\T[D]$;
if $C$ and $C'$ were parallel, then $T[C]$ and $T[C']$
would intersect transversely along $D$ and $D^\perp$ only,
as in the preceding paragraph, so they must intersect orthogonally.
By identical reasoning $C^\perp$ and $C'^\perp$ must also intersect
orthogonally and in fact all four axis-circles
lie on $\T[D']$ as well as on $\T[D]$,
whose intersection is therefore precisely the union of these circles.
Conversely, if we assume $C$ and $C'$ are two orthogonally intersecting great circles,
then we can construct the two Clifford tori $\T$ and $\T'$ containing $C$ and $C'$
(so that $\T \cap \T'=C \cup C^\perp \cup C' \cup C'^\perp$)
and their corresponding axes $D \cup D^\perp$ and $D'\cup D'^\perp$;
it follows from the previous three sentences
(exchanging the roles of all $C$ and $D$ circles)
that $D$ and $D'$ as well as $D^\perp$ and $D'^\perp$
intersect orthogonally and their union is all of $T[C] \cap T[C']$.

Starting with orthogonally intersecting axis-circles $C$ and $C'$ as above,
by varying the angle between $C$ and $C'$ while fixing their two intersection points 
one obtains pairs of tori intersecting along (noncircular) curves with
four points of tangency. The full space of intersections is two-dimensional, corresponding
to the relative configuration of representative axis-circles. It can
be parametrized by, for instance, the minimum and maximum distances
from one of the axis-circles to points on the other. The case of tangency
corresponds to the minimum vanishing, while the case of constant-angle
intersection along a pair of orthogonal great circles corresponds to equality
of the minimum and maximum.

We remark finally that any two Clifford tori necessarily intersect.
This is a consequence of a general result of Frankel \cite{Fra},
but for an elementary argument note that every single Clifford torus divides $\Sph^3$
into two connected components, 
which are open tubular neighborhoods of each axis-circle,
so if two Clifford tori failed to intersect, then one side of one of
them would have to be properly contained in one side of the other,
but the two sides have equal volume, precluding this situation.

\subsection*{The initial configurations}
Although in the previous subsection we discussed Clifford tori in some generality, 
in this article we will only be concerned with 
finite maximally symmetric subcollections of 
a one-parameter family of tori intersecting
along a pair of totally orthogonal great circles and the Clifford torus equidistant from the circles of intersection. 
To fix the notation 
we start by taking $\Sph^3$ to be the unit sphere in $\R^4$,
to which we give its standard orientation and which we will routinely identify
with $\C^2$ via the map $(x^1,y^1,x^2,y^2) \mapsto (x^1+iy^1,x^2+iy^2)$,
so that $\Sph^3=\left\{ (z_1,z_2) \in \C^2 \, : \, \abs{z_1}^2+\abs{z_2}^2=1 \right\}$.
We write
  \begin{equation}
    C_1:=\left\{ \left(e^{it},0\right) \; : \; t \in \R \right\}
      \quad \mbox{ and } \quad
    C_2:=\left\{ \left(0,e^{it}\right) \; : \; t \in \R \right\} = C_1^\perp 
  \end{equation}
for the unit circles in the coordinate planes, oriented by increasing $t$. 
As described in the previous subsection there is a 
unique Clifford torus
  \begin{equation}
    \T':=\T[C_1]=\T[C_2]
    =\left\{ \frac{1}{\sqrt{2}}\left(e^{i\xx}, e^{i\yy}\right): \xx, \yy \in \R \right\}
  \end{equation}
equidistant from $C_1$ and $C_2$.
There is also a
one-parameter family of Clifford tori containing $C_1$ (so also $C_2$)
and right-handed along it (so also along $C_2$);
we distinguish the one
  \begin{equation}
    \T := \left\{ e^{i \zz} (\cos \rr, \sin \rr): \rr, \zz \in \R \right\}
  \end{equation} 
that contains the great circle $C_{0,0}$ in the real plane $\{\Im z_1=\Im z_2=0\}$ in $\C^2$ 
(see \ref{E:Cphph} for the notation).  

For each integer $k \geq 2$ we intend to desingularize the configurations
\begin{equation} 
\label{E:Wk} 
\Wcal_k := \textstyle\bigcup_{j=1}^{k} \T_j
\quad \mbox{ and } \quad
\Wcal_k':=\Wcal_k \textstyle\bigcup \T'
\end{equation}
of $k$ and $k+1$ Clifford tori respectively,
where for each integer $j$
  \begin{equation}
\label{E:Cphph} 
    \T_j := \rot_{C_1}^{(j-1)\pi/k}\T
    =\rot_{C_2}^{-(j-1)\pi/k}\T
    =\left\{ e^{i \zz} \left(\cos \rr, \, e^{i\frac{j-1}{k}\pi} \sin \rr\right): \rr, \zz \in \R \right\},
  \end{equation}
so that $\T_1=\T$ and $\T_{j+k}=\T_j$ for all $j \in \Z$.
Thus $\T_j$ intersects $\T_\ell$ at constant angle $\frac{j-\ell}{k}\pi$
along $C_1$ and $C_2$,
while $\T_j$ intersects $\T'$ orthogonally along the two totally orthogonal geodesics on $\T_j$
equidistant to $C_1$ and $C_2$. 
 
Reflections through certain great circles will play an important role in the construction,
so we name some of these circles now.
First we define
for each $j \in \frac{1}{2}\mathbb{Z}$
  \begin{equation}
\label{E:Cp} 
    C'_j := \rot_{C_1}^{(j-1) \frac{\pi}{k}} 
      \left\{ \left. \frac{e^{i\zz}}{\sqrt{2}}(1,1) \; \right| \zz \in \R \right\}
      = \left\{ \left. \frac{e^{i\zz}}{\sqrt{2}}\left(1,e^{i\frac{j-1}{k}\pi}\right)
        \; \right| \zz \in \R \right\}
    \subset \T',
  \end{equation}
oriented by increasing $\zz$.
Clearly for each $j \in \frac12\Z$ 
  \begin{equation}
\label{E:Cpperp}  
{C'_j}^\perp= C'_{j+k} 
    = \rot_{C_1}^{\pi} C'_j 
    =\rot_{C_2}^{\pi} C'_j, 
\end{equation} 
so in total there are $4k$ such great circles, pairwise disjoint;
$\{C'_j\}_{j \in \Z}$ consists
of the $2k$ great circles where $\T'$ intersects the other Clifford tori in $\Wcal_k'$:
\begin{equation}
\T_j\cap\T'=C'_j \cup C'_{j+k} 
=C'_j \cup {C'_{j}}^\perp    
\end{equation}
and $C'_j$ and $C'_{j+k}$ are parallel on $\T_j$ to $C_1$ and $C_2$ at a distance $\pi/4$; 
and for $j\in\Z$ $C'_{j+1/2}$ is the closer of two totally orthogonal
great circles on $\T'$ equidistant from $C'_j$ and $C'_{j+1}$.
We also mention that
\begin{equation} 
\label{E:Cpinter}
\T_j= \T[C'_{j+k/2}] = \T[C'_{j+3k/2}].
\end{equation}

Next for $\phi_1,\phi_2 \in \R$ we label the great circle orthogonally
intersecting $C_1$ at $\pm (e^{i\phi_1},0)$ and $C_2$ at $\pm (0,e^{i\phi_2})$ by
  \begin{equation}
    C_{\phi_1,\phi_2} = \left\{ \left. \left(e^{i\phi_1}\cos \rr, e^{i\phi_2}
      \sin \rr \right)
      \; \right| \; \rr \in \R \right\}.
  \end{equation}
Thus $C_{\phi_1,\phi_2}$ and $C_{\phi_1',\phi_2'}$
are disjoint unless $\phi_1-\phi_1'$ or
$\phi_2-\phi_2'$ is an integral multiple of $\pi$,
in which case they intersect only at two antipodal points on $C_1$
(when $\phi_2-\phi_2' \not \in \pi\Z$),
only at two antipodal points on $C_2$
(when $\phi_1-\phi_1' \not \in \pi\Z$),
or they coincide;
in particular $C_{\phi_1+j_1\pi, \, \phi_2+j_2\pi}=C_{\phi_1,\phi_2}$
for all $j_1,j_2 \in \Z$.
Note also that 
\begin{equation}
C_{\phi_1,\phi_2}^\perp = C_{\phi_1+\pi/2, \, \phi_2+\pi/2}, 
\qquad 
C_{\phi_1,\phi_2} \cap \T' = \{ \, 2^{-1/2}(\pm e^{i\phi_1}, \pm e^{i\phi_2}) \,\} 
\end{equation} 
where the intersections at the four points are orthogonal, 
and that $\T_j$ is foliated by the disjoint great circles $C_{\phi_1,\phi_2}$
satisfying
$\phi_2=\phi_1+\frac{j-1}{k}\pi$ for $\phi_1 \in [0,\pi)$.

Last we define for each $\psi \in \R$ the great circle
  \begin{equation}
    C''_{\psi} \, := \, \left\{ \left. \frac{1}{\sqrt{2}} 
      \left(e^{i\zz},e^{i(\psi-\zz)}\right) \; \right| \; \zz \in \R \right\}
    =\left\{ \left. \frac{e^{i\zz}}{\sqrt{2}}
      \left(1,e^{i(\psi-2\zz)}\right) \; \right | \; \zz \in \R \right\} \subset \T',
  \end{equation}
oriented by increasing $\zz$.
Note that 
the Clifford torus $\T'$ is foliated by the disjoint great circles
$C''_\psi$ with $\psi \in [0,2\pi)$, 
\begin{equation} 
C''_\psi \cup {C''}_\psi^\perp =\T' \cap \rot_{C_1}^\psi \rot_{C_{0,\pi/2}}^{\pi/2}\T,
\qquad 
{C''_\psi}^\perp=C''_{\psi+\pi}, 
\end{equation} 
and the parameter $\psi$ measures the angle at the point $(1,0)$ between $\T$
and $\rot_{C_1}^\psi \rot_{C_{0,\pi/2}}^{\pi/2}\T$,  
which is a Clifford torus through 
$C_1$ and $C_2$ but left-handed along both. 
Note also that the circles $C''_\psi$ and $C'_j$ intersect orthogonally at two points 
given by 
\begin{equation} 
C''_\psi \cap C'_j 
=
\left\{ 
\pm \frac{1}{\sqrt{2}}\left(e^{i\left(\frac{\psi}{2}-\frac{j-1}{2k}\pi\right)},
e^{i\left(\frac{\psi}{2}+\frac{j-1}{2k}\pi\right)}\right)
\right\}.   
\end{equation} 

\subsection*{The symmetries of the initial configurations and the main constructions}
For future reference we first note the identity
  \begin{equation}
    \rot_C^{\theta}\rot_{C'}^{\theta'} \rot_C^{-\theta}
    =
    \rot_{\rot_C^\theta C'}^{\theta'}
  \end{equation}
for any two great circles $C$ and $C'$ and angles $\theta, \, \theta' \in \R$, 
as well as the particular products
  \begin{equation}
  \label{symcomp}
    \begin{aligned}
      &\rot_{C_{\phi_1,\phi_2}}^\pi \rot_{C_{\phi'_1,\phi'_2}}^\pi
        = \rot_{C_1}^{2(\phi_2-\phi'_2)}\rot_{C_2}^{2(\phi_1-\phi'_1)},  \\
      &\rot_{C'_j}^\pi \rot_{C'_\ell}^\pi
        = \rot_{C_1}^{\frac{\pi}{k}(j-\ell)}\rot_{C_2}^{\frac{\pi}{k}(\ell-j)}, \\
      &\rot_{C''_{\psi_1}}^\pi \rot_{C''_{\psi_2}}^\pi
        = \rot_{C_1}^{\psi_1-\psi_2} \rot_{C_2}^{\psi_1-\psi_2}, \mbox{ and } \\
      &\rot_{C_{0,0}}^\pi \rot_{C'_1}^\pi
        = \rot_{C''_0}^\pi.
    \end{aligned}
  \end{equation}

We now describe the symmetry groups of the above configurations.

\begin{lemma}
\label{configsym} 
(i) For $k \geq 2$  
  \begin{equation} 
  \label{GsymWk}
    \Gsym(\Wcal_k)
    = 
    \left \langle
      \left\{ \rot_{C_1}^{\pi/k}, \, \rot_{C_{0,0}}^\pi, \, \rot_{C_1'}^\pi \right\}
      \cup \left\{\rot_{C_1}^\alpha \rot_{C_2}^\alpha \right\}_{\alpha \in \R} \right \rangle.
  \end{equation}
In particular this group contains reflection through $C_{\phi_1,\phi_2}$
whenever $\phi_1-\phi_2 \in \frac{\pi}{2k}\mathbb{Z}$,
through $C'_j$ for all $j \in \frac{1}{2}\mathbb{\Z}$,
and through $C''_\psi$ for all $\psi \in \R$. 

(ii) $\Gsym(\Wcal_k')=\Gsym(\Wcal_k)$ for $k>2$. 

(iii) $\Gsym(\Wcal_2')=
  \left \langle
        \left\{ \rot_{C_1}^{\pi/2}, \, \rot_{C_{0,0}}^\pi, \, \rot_{C_1'}^{\pi/2} \right\}
        \cup \left\{\rot_{C_1}^\alpha \rot_{C_2}^\alpha \right\}_{\alpha \in \R} \right \rangle
  \supsetneq \Gsym(\Wcal_2)$. 

\end{lemma}

\begin{proof}
First suppose $k \geq 2$.
It is clear that $\Gsym(\Wcal_k)$ contains
the group on the right-hand side of equation \ref{GsymWk}.
A symmetry of $\Wcal_k$
will either exchange the intersection circles $C_1$ and $C_2$ or will preserve each as a set.
In particular it will preserve $\T'=\T[C_1]=\T[C_2]$,
so $\Gsym(\Wcal_k) \subseteq \Gsym(\Wcal_k')$.
Moreover, each element of $O(4)$ is completely determined by its action on $C_1$ and $C_2$.

If a symmetry of $\Wcal_k$ preserves $C_1$, then it also preserves $C_2$
and it must act by rotation or reflection in each circle.
An orthogonal transformation reversing the orientation of one circle
but preserving the orientation of the other 
cannot be a symmetry of $\Wcal_k$
(since $C_1$ and $C_2$ are right-handed in all of the tori of the configuration and a reflection in just one of them reverses chirality). 
A symmetry acting by a reflection in one circle must therefore act by reflection in both circles; 
in other words such a symmetry must be a reflection through a geodesic of $\Sph^3$ orthogonally intersecting $C_1$ and $C_2$.  
Reflection through a geodesic orthogonal to both $C_1$ and $C_2$ will preserve the configuration 
precisely when the geodesic lies in $\Wcal_k \cup \rot_{C_1}^{\frac{\pi}{2k}}\Wcal_k$, 
meaning on one of the tori or halfway between two tori.
Reflections through these geodesics clearly belong
to the group on the right-hand side of \ref{GsymWk}.
Any rotation in each circle preserving $\Wcal_k$
can be composed with a symmetry of the form
$\rot_{C_1}^\alpha\rot_{C_2}^\alpha$
to produce a symmetry fixing $C_1$ pointwise and rotating $C_2$.
Clearly any such symmetry likewise belongs to the group on the right-hand side of \ref{GsymWk},
so we have now accounted for all symmetries of $\Wcal_k$ preserving $C_1$ and $C_2$ separately. 
 
Any transformation exchanging the two circles is the product of a transformation preserving them 
with the reflection through a geodesic equidistant from $C_1$ and $C_2$. 
These are the great circles on the torus $\T'$, 
and reflection through such a circle is a symmetry of the configuration precisely 
when the geodesic either orthogonally intersects all the other tori or else lies on 
either one of the other tori or halfway between a consecutive pair of these.
All these reflections belong to the right-hand side of \ref{GsymWk} too,
so we have finished checking (i). 

If $k>2$, then any element of $\Wcal_k'$
must permute the intersection circles $C_1$ and $C_2$
(and cannot exchange them with any intersection circles on $\T'$)
and therefore must preserve $\Wcal_k$, confirming (ii).
In the case $k=2$, however,
the three tori $\T_1$, $\T_2$, and $\T'$ are all equivalent under the symmetries,
as are all six intersection circles $C_1$, $C_2$, $C_1'$, $C_2'$, $C_3'$, and $C_4'$.
In particular $\rot_{C_1'}^{\pi/2}$ belongs to $\Gsym(\Wcal'_2)$
and exchanges $\T_1$ and $\T'$.
Any symmetry of $\Wcal'_2$ exchanging $\T'$ with either of the other Clifford tori
can be composed with $\rot_{C_1'}^{\pi/2}$ and (possibly) an element
of $\Gsym(\Wcal_2)$ to obtain an element of $\Gsym(\Wcal_2)$,
completing the proof of (iii). 
\end{proof}

The choices we are about to make to desingularize the initial configurations will break many
of the symmetries just described, including in particular the continuous symmetries.
Nevertheless, to simplify the construction we will insist on
retaining reflections through
a collection of great circles which will be included in their entirety
on the surfaces we construct
and which serve as a sort of scaffolding for the construction.
More precisely for any integers $k \geq 2$ and $m \geq 1$
we introduce the scaffoldings
  \begin{equation}
    \label{scaff}
      \mathcal{C}_{k,m} := \bigcup_{j, \ell \in \Z} C_{\frac{j\pi}{k,m},\frac{j\pi}{km}+\frac{\ell \pi}{k}}
        \subset \Wcal_k
      \quad \mbox{ and } \quad
      \mathcal{C}'_{k,m} := \mathcal{C}_{k,2m} \cup
        \bigcup_{j \in \Z} 
        C''_{\frac{j\pi}{km}} \subset \Wcal'_k
  \end{equation}
and corresponding groups
  \begin{equation}
  \label{symm}
    \Grp_{k,m} := \Grefl(\mathcal{C}_{k,m}) \subset O(4)
    \quad \mbox{ and } \quad
    \Grp'_{k,m} := \Grefl(\mathcal{C}'_{k,m}) \subset O(4).
  \end{equation}
Motivation for the choice of scaffolds
and corresponding symmetries can be found in the next section.

\begin{lemma}
For any integers $k \geq 2$ and $m \geq 1$
  \begin{enumerate}[(i)]
    \item $\Ckm$ is the union of $k^2m$ great circles, 
      of which the $km$ circles
        $C_{\frac{j\pi}{km}, \frac{j\pi}{km}+\frac{\ell-1}{k}\pi}$
         with $1 \leq j \leq km$ are parallel on $\T_\ell$;
    \item $\Ckmp$ is the union of $2k^2m+2km$ great circles,
      of which the $2km$ circles
      $C_{\frac{j\pi}{2km}, \frac{j\pi}{2km}+\frac{\ell-1}{k}\pi}$
      with $1 \leq j \leq 2km$ are parallel on $\T_\ell$
      and the $2km$ circles
      $C''_{\frac{j\pi}{km}}$ with $1 \leq j \leq 2km$ are parallel on $\T'$;
    \item $\Ckm$ intersects each of $C_1$ and $C_2$ at the $2km\textsuperscript{th}$ roots of unity,
      at each of which points exactly $k$ great circles in $\Ckm$ intersect,
        one on each $\T_\ell$;
    \item $\Ckmp$ intersects each of $C_1$ and $C_2$ at the $4km\textsuperscript{th}$ roots of unity,
      at each of which points exactly $k$ great circles in $\Ckmp$ intersect,
        one on each $\T_\ell$;
    \item $\Ckmp$ intersects each $C'_\ell$ with $\ell \in \Z$
      at the $4km$ points
      $\frac{1}{\sqrt{2}}\left(e^{i\frac{j\pi}{2km}},e^{i\frac{j\pi}{2km}+i\frac{\ell-1}{k}\pi}\right)$,
      with $1 \leq j \leq 4km$,
      at each of which exactly two great circles in $\Ckmp$ intersect,
      namely $C_{\frac{j\pi}{2km},\frac{j\pi}{2km}+\frac{\ell-1}{k}\pi}$ on $\T_\ell$
      and $C''_{\frac{j\pi}{km}+\frac{\ell-1}{k}\pi}$ on $\T'$; 
    \item $\Grp_{k,m}  
= 
      \left \langle \rot_{C_{0,0}}^\pi, \,
      \rot_{C_1}^{\frac{2\pi}{km}}\rot_{C_2}^{\frac{2\pi}{km}}, \, 
      \rot_{C_1}^{\frac{2\pi}{k}} \right\rangle
= 
      \left \langle \rot_{C_{0,0}}^\pi, \,
      \rot_{C_1}^{\frac{2\pi}{km}}\rot_{C_2}^{\frac{2\pi}{km}}, \, 
      \rot_{C_2}^{\frac{2\pi}{k}} \right\rangle
      \subset \Gsym(\Ckm) \cap \Gsym(\Wcal_k)$; and
    \item  $\Grp'_{k,m} =
      \left \langle 
      \Grp_{k,2m}, \, 
      \rot_{C'_1}^\pi  \right \rangle
      \subset \Gsym(\Ckmp) \cap \Gsym(\Wcal'_k)$.
  \end{enumerate} 
\end{lemma}

\begin{proof} 
Items (i)-(v) follow immediately from the definitions,
as do the containments in the last two items.
Referring to \ref{symcomp}, we observe
  \begin{equation}
\begin{aligned} 
    \rot_{C_{\frac{j\pi}{2km},\frac{j\pi}{2km}+\frac{\ell \pi}{k}}}^\pi
      \rot_{C_{0,0}}^\pi
   & =\left(\rot_{C_1}^{2\pi/k}\right)^\ell
      \left(\rot_{C_1}^{\pi/km}\rot_{C_2}^{\pi/km}\right)^j,
\\ 
    \rot_{C_{0,0}}^\pi \rot_{C'_1}^\pi
&=\rot_{C''_0}^\pi,
\\ 
    \rot_{C''_0}^\pi \rot_{C''_{\frac{j\pi}{km}}}^\pi
   & =
    \left(\rot_{C_1}^{\pi/km}\rot_{C_2}^{\pi/km}\right)^j, 
\end{aligned} 
  \end{equation}
completing the proof.
\end{proof}

\section{The initial surfaces}
\label{Sinit}

\subsection*{Toral coordinates along a great circle} 

We define $\Phi : \R^3 \to \Sph^3$ by 
\begin{equation}
\label{E:Phi} 
\Phi(\rr \cos \theta, \rr \sin \theta, \zz) = e^{i\zz}(\cos \rr, e^{i\theta} \sin \rr) . 
\end{equation}
Observe that $\Phi$ takes planes 
of constant $\theta$ to Clifford tori through $C_1$ and $C_2$, 
cylinders of constant $\rr$ to constant-mean-curvature tori 
(degenerating to $C_1$ or $C_2$ when $\rr$ is respectively an even or odd multiple of $\pi/2$)
with axis-circles $C_1$ and $C_2$,
horizontal planes to great two-spheres with equator $C_2$,
and radial and vertical lines to great circles. 
In particular
$\Phi(\{\yy=0\})=\T$,
$\Phi\left(\left\{\sqrt{\xx^2+\yy^2}=\frac{\pi}{4}\right\}\right)=\T'$,
$\Phi(\{\xx=\yy=0\})=C_1$,
$\Phi(\{\yy=\zz=0\})=C_{0,0}$,
and $\Phi\left(\left\{\xx=\frac{\pi}{4}, \, \yy=0\right\}\right)=C'_1$.
On the other hand
the Clifford torus left-handed along $C_1$
and containing the circle $\rot_{C_1}^\psi C_{0,0}$
is the image under $\Phi$ of the helicoid $\theta=\psi-2\zz$,
and $C''_\psi$ and $C''_{\psi+\pi}$ are the images
of the two helices where this helicoid intersects
the cylinder $\rr=\pi/4$.
In particular
$\Phi\left(\left\{
\left(\frac{\pi}{4} \cos \theta, \, \frac{\pi}{4} \sin \theta, \, -\frac{1}{2}\theta \right) \; : \;
  \theta \in [0,2\pi) \right\} \right)=C''_0$.
Though great spheres do not play an important role in the construction,
to aid understanding of the geometry of $\Phi$ we also mention that
the great sphere with equator $C_1$
and poles $\pm (0,e^{i\phi})$
is the image of the helicoid $\theta=\phi-\zz$.

Moreover, the open solid torus
$\left\{\sqrt{\xx^2+\yy^2}<\frac{\pi}{4}\right\}
/ \left\langle \transe_{\zz \text{-axis}}^{2\pi} \right\rangle$ 
is mapped diffeomorphically by $\Phi$ onto the open 
solid torus of points within $\frac{\pi}{4}$ of $C_1$---the $C_1$ side of $\T'$---and 
$\Phi$ is an approximate isometry for small $\rr$.
More precisely, writing $\gsph$ for the round metric on $\Sph^3$
and $\geuc$ for the Euclidean metric on $\R^3$, we have
\begin{equation}
\label{phullback}
  \Phi^*\gsph = 
    d\rr^2 + \sin^2 \rr \, d\theta^2 + 2 \sin^2 \rr \, d\theta \, d\zz + d\zz^2 \\
    = \geuc + \left(\sin^2 \rr - \rr^2\right)d\theta^2 
      + 2 \sin^2 \rr \, d\theta \, d\zz.
\end{equation}
We also note that $\Phi$ intertwines some symmetries of interest: for every $c\in\R$
\begin{equation}
\label{phintertwine}
  \begin{aligned}
    &\Phi \rote_{\zz \text{-axis}}^{c} = \rot_{C_1}^{c} \Phi, \\
    &\Phi \rote_{\xx \text{-axis}}^\pi = \rot_{C_{0,0}}^\pi \Phi, \text{ and} \\
    &\Phi \transe_{\zz \text{-axis}}^c
      = \rot_{C_1}^c \rot_{C_2}^c \Phi.
  \end{aligned}
\end{equation}

\subsection*{Towers with straightened wings}
Since $\Phi$ maps the asymptotic planes of the Karcher-Scherk towers to Clifford tori, 
we can smoothly glue the towers imported by $\Phi$ to the tori of our configurations
by straightening the wings to exact half-planes before applying $\Phi$.
We accomplish this transition using the cut-off functions defined in \ref{Epsiab}.
Specifically, 
given integers $k \geq 2$ and $m \geq 1$,
we set 
  \begin{equation}
  \label{adef}
    a:=a_m:=\frac{m\pi}{4}-10 
  \end{equation}
and we define the map
  \begin{equation}
    \xkm: \Sk \to \R^3
  \end{equation}
by modifying the inclusion map
$\iota_{_{\Sk}}: \Sk \to \R^3$
(recall \ref{tower}) as follows.
We will specify $\xkm$ on the wedge 
$E:=\{(\rr \cos \theta, \rr \sin \theta, \zz) \; : \;
  \rr \geq 0, \, \abs{\theta} \leq \pi/2k, \, \zz \in \R \}$ 
and complete its global definition by requiring it
to commute with all elements
of $\Gsym(\Sk)$. 
We assume that $m$ is large enough so that $a>R_k$.
The new map $\xkm$ agrees with $\iota_{_{\Sk}}$
on 
$\Sk \cap E \cap \{\xx<a\}$.
We recall from \ref{tower} that the complement
$\Sk \cap E \cap \{\xx \geq a\}$
is the graph $\{(\xx,W_k(\xx,\zz),\zz) \; : \; \xx \geq a, \, \zz \in \R\}$
of the function $W_k$ over the half-plane $[a,\infty) \times \{0\} \times \R$
and further that this graph misses the boundary of the wedge.
We decree that
  \begin{equation}
  \label{xkm}
    \xkm: (\xx,W_k(\xx,\zz),\zz) \mapsto (\xx,\cutoff{a+1}{a}(\xx)W_k(\xx,\zz),\zz)
  \end{equation}
for each point $(\xx,W_k(\xx,\zz),\zz)$ in this region.
Imposing the symmetries as just described completes the definition of $\xkm: \Sk \to \R^3$,
and we also define its truncated image
  \begin{equation}
  \label{Stildekm}
    \Stildekm(s) = \xkm \left(\Sk \cap \{ \sqrt{\xx^2+\yy^2} \}\right)
  \end{equation}
for any $s>0$ controlling the truncation.

\subsection*{Data, Symmetries, and Scaffolding}
Each initial surface $M$ desingularizing $\Wcal_k$
is specified by a quintuple of data $(k,m,n_1,n_2,\sigma)$,
and each initial surface $N$ desingularizing $\Wcal'_k$
is specified by a 
septuple $(k,m,n,n'_1,n'_{-1},\sigma'_1,\sigma'_{-1})$,
where $k \geq 2$ and $m,n,n_1,n_2,n'_1,n'_{-1} \geq 1$ are integers
and $\sigma,\sigma'_1,\sigma'_{-1} \in \{0,1\}$.
The role of $k$ in the definition of the initial configurations is already clear:
it is the number of Clifford tori intersecting along $C_1$ and $C_2$.
Along with $m$ it determines the scaffolding ($\Ckm$ or $\Ckmp$)
according to \ref{scaff}.
We have also used $m$ in \ref{adef}
to set the distance from a tower's axis at which its wings are straightened.
Together with $k$ and $m$ the remaining positive integers
(either $n$, $n_1$, and $n_2$ or $n'_1$ and $n'_{-1}$) specify
the number of fundamental periods, or equivalently the scale,
of the towers along each circle of intersection.
Finally, the relative alignment of towers along the various circles
is prescribed by either $\sigma$ or $\sigma'_1$ and $\sigma'_{-1}$.

We will now explain the roles of the data in slightly more detail
and simultaneously offer some brief motivation
for the definitions of the scaffoldings (\ref{scaff} above)
and the initial surfaces (\ref{initsdef} below).
The initial surfaces are to be constructed from the initial configurations
they are intended to desingularize
by replacing a tubular neighborhood of each intersection circle $C$
with a truncated Karcher-Scherk tower with straightened wings,
scaled in $\R^3$ by some factor $\mC>0$,
mapped into $\Sph^3$ by $\Phi$, and then positioned
as desired along $C$ by a rotation $\rot[C] \in O(4)$:
  \begin{equation}
  \label{SigmaC}
    \SigmaC:=\rot[C]\Phi
    \left(
      \mC^{-1}
      \widetilde{\mathcal{S}}_{\kC,m}\left(\mC s_{_C}\right)
    \right),
  \end{equation}
where $\kC$ is the number (either $k$ or $2$) of Clifford tori intersecting along $C$
in the corresponding initial configuration
and where $s_{_C}>0$ is picked, somewhat arbitrarily,
to ensure each tower is truncated well away from towers on neighboring circles.

Since $\Sk$ has fundamental period $2\pi$
and $\Phi$ is periodic in $\zz$ with period $2\pi$
(each circle of intersection having length $2\pi$),
obviously $\mC$ must be an integer
in order for the initial surface to be embedded.
Further constraints on each $\mC$ are placed
by the following symmetry requirements which we make of the initial surface
to simplify the analysis of the linearized operator on the towers in Section \ref{linear}.
  \begin{assumption}
    \label{sa}
    Let $\Sigma$ be an initial surface.
    We require $C_{0,0} \subset \Sigma$ and
    $\rot_{C_{0,0}}^\pi \in \Gsym(\Sigma)$,
    and for each intersection circle $C$
    in the corresponding initial configuration
    we require $\rot_C^{2\pi/\kC} \in \Gsym(\Sigma)$.
  \end{assumption}
\noindent The assumption ensures triviality of the kernel of each tower's Jacobi operator
restricted to the space of deformations respecting $\Gsym(\Sigma)$;
we do not claim that these conditions are necessary, but they are natural and sufficient.

It is not hard to check (using \ref{tower} and \ref{symcomp} for instance)
that imposing \ref{sa} is equivalent to demanding
$\Ckm \subset \Sigma$ and $\Grefl(\Ckm) \subset \Gsym(\Sigma)$
(when $\Sigma$ desingularizes $\Wcal_k$)
or $\Ckmp \subset \Sigma$ and $\Grefl(\Ckmp) \subset \Gsym(\Sigma)$
(when $\Sigma$ desingularizes $\Wcal'_k$)
for some integer $m \geq 1$.
In particular each tower $\SigmaC$ must itself contain the appropriate scaffolding,
forcing $\mC$ to divide $km$ (when $\Ckm \subset \Sigma$)
or $2km$ (when $\Ckmp \subset \Sigma)$.
Since $\Grp_{k,m}$ includes no symmetries exchanging $C_1$ and $C_2$,
the quotients $m_{_{C_1}}/km$ and $m_{_{C_2}}/km$ are independent
and are given by $n_1$ and $n_2$.
On the other hand modulo $\Grp'_{k,m}$
$C_1$ and $C_2$ are equivalent to one another
but to no other intersection circles,
while for every $j \in \Z$
the circles $C'_j$ and $C'_{j+2}$ are equivalent
but $C'_j$ and $C'_{j+1}$ are inequivalent.
Thus only $m_{_{C_1}}/2km$,
$m_{_{C'_1}}/2km$, and $m_{_{C'_2}}/2km$
can be independently prescribed as $n$, $n'_1$, and $n'_{-1}$ respectively.

\subsection*{Alignment}
The data so far described completely determine the periods and sizes
of the towers replacing the circles of intersection,
but this information and the particular scaffolding
do not quite fix the initial surface $\Sigma$, even up to congruence.
(We call two initial surfaces $\Sigma_1$ and $\Sigma_2$
\emph{congruent} in $\Sph^3$ if there exists $\rot \in O(4)$
such that $\Sigma_2=\rot \Sigma_1$.)
Specifically there is some not entirely inconsequential freedom
in the choice of $\rot[C]$ in \ref{SigmaC}:
it may be replaced by $\rot_C^{\pi/\kC}\rot[C]$;
equivalently we could replace the tower $\Sk$ defining $\widetilde{\mathcal{S}}_{\kC,m}$
by its ``dual'' tower mentioned in Remark \ref{towerremark}
and having the same scaffolding and symmetry group as $\Sk$
but occupying $A_k'$ instead of $A_k$.

Accordingly we allow the values of $\sigma$
to make a choice of one model tower or the other along $C_2$
for initial surfaces desingularizing $\Wcal_k$
and we use $\sigma'_1$ and $\sigma'_{-1}$ for identical purposes
along the towers $C'_j$ with $j$ odd and even respectively
for initial surfaces desingularizing $\Wcal'_k$.
Thus these data control the \emph{alignment} of the towers.
It is not necessary to allow for the two possibilities on every (inequivalent) intersection circle,
since up to congruence it is only the relative alignment that matters.
In fact, even allowing realignment on just the circles mentioned,
we still sometimes produce congruent initial surfaces with different values of $\sigma$
(or $\sigma'_1$ and $\sigma'_{-1}$), depending on the parities of the other data.

For example, assuming $n_1$ and $n_2$ relatively prime
(since we may absorb a common divisor into $m$),
the surfaces $M(k,m,n_1,n_2,\sigma=0)$ and $M(k,m,n_1,n_2,\sigma=1)$
(formally defined in \ref{initsdef} below)
desingularizing $\Wcal_k$ are congruent precisely when
either $m$ or exactly one $n_i$ is odd.
Indeed, if $mn_1$ is odd, then
  \begin{equation}
    \rot_{C_2}^{\frac{\pi}{k}}M(k,m,n_1,n_2,0)
    =
    \left(\rot_{C_1}^{\frac{\pi}{kmn_1}}\rot_{C_2}^{\frac{\pi}{kmn_1}}\right)^{mn_1}
      \rot_{C_1}^{-\frac{\pi}{k}}M(k,m,n_1,n_2,0)
    =
    M(k,m,n_1,n_2,1);
  \end{equation}
if $n_1$ is even but $n_2$ odd, then
  \begin{equation}
    \left(\rot_{C_1}^{\frac{\pi}{kmn_2}}\rot_{C_2}^{\frac{\pi}{kmn_2}}\right)^{n_2}M(k,m,n_1,n_2,0)
    =
    \left(\rot_{C_1}^{\frac{\pi}{kmn_1}}\rot_{C_2}^{\frac{\pi}{kmn_1}}\right)^{n_1}
      M(k,m,n_1,n_2,0)
    =
    M(k,m,n_1,n_2,1);
  \end{equation}
and if $n_1$ is odd but $m$ and $n_2$ even, then
  \begin{equation}
    \rot_{C_2}^{\frac{\pi}{k}}
      \left(\rot_{C_1}^{\frac{\pi}{kmn_2}}\rot_{C_2}^{\frac{\pi}{kmn_2}}\right)^{n_2}
      M(k,m,n_1,n_2,0)
    =
    \left(\rot_{C_1}^{\frac{\pi}{kmn_1}}\rot_{C_2}^{\frac{\pi}{kmn_1}}\right)^{(m+1)n_1}
      \rot_{C_1}^{-\frac{\pi}{k}}
      M(k,m,n_1,n_2,0)
    =
    M(k,m,n_1,n_2,1).
  \end{equation}

To see that $M(k,m,n_1,n_2,0)$ and $M(k,m,n_1,n_2,1)$
are not congruent when $m$ is even and both $n_1$ and $n_2$ are odd,
recall that each initial surface
will be defined to
include the scaffold circles whose union is $\mathcal{C}_{k,m}$,
which orthogonally intersect $C_1$ and $C_2$
at the $2km\textsuperscript{th}$ roots of unity on each.
Choosing the global normal which at $(1,0)$
points in the positive direction of $C_1$ then determines
the direction of that normal at all $2km\textsuperscript{th}$ roots of unity
on $C_1$ and $C_2$.
We call the unit normal at such a point positive
if it points in the positive direction of the circle on which it lies
and we call it negative otherwise.
Then, given two such points joined
by an arc of a great circle in $\Ckm$,
we say that the two normals there
are aligned if they are either both positive or both negative 
and otherwise say that they are antialigned.
The parities assumed for the data imply that,
for a given initial surface with that data,
this alignment does not depend on the pair of chosen points
and therefore
defines a property of the initial surface
which is invariant under congruences that preserve $\Ckm$,
but it is reversed by altering the value of $\sigma$.
This shows that $M(k,m,n_1,n_2,0)$ is not congruent to $M(k,m,n_1,n_2,1)$.

Similar considerations apply to surfaces desingularizing $\Wcal'_k$,
but to avoid complicating the definition of the initial surfaces
we do not make a systematic effort to eliminate completely
duplication of congruence classes within the collection of initial surfaces.

\subsection*{Definition and basic properties}
With the foregoing in mind we define the initial surfaces as follows,
recalling \ref{adef} and \ref{Stildekm}.
\begin{definition}
\label{initsdef}
Given integers $k \geq 2$, $m \geq 1$,
and relatively prime $n_1,n_2>0$,
as well as $\sigma \in \{0,1\}$,
set
\begin{equation}
  M(k,m,n_1,n_2,\sigma) 
    :=
   \bigcup_{j=1}^2 
     \left(\rot_{C_2}^{\frac{\pi}{k}}
       \right)^{(j-1)\sigma}
       \rot_{C'_1}^{(j-1)\pi} \Phi \left( \frac{1}{kmn_j}
       \Stildekm \left(kmn_j\frac{\pi}{4}\right) \right).
\end{equation}

Given instead integers $k \geq 2$, $m\geq1$, and 
relatively prime $n,n'_1,n'_{-1}>0$,
as well as $\sigma'_1,\sigma'_{-1} \in \{0,1\}$,
set
\begin{equation}
  \begin{aligned}
  &N(k,m,n,n'_1,n'_{-1},\sigma'_1,\sigma'_{-1})
  :=
    \bigcup_{j=0}^1 
    \rot_{C'_1}^{j \pi} \Phi \left(\frac{1}{2kmn}
    \Stildekm \left(2kmn\left[\frac{\pi}{4}-\frac{\pi}{4k}\right] \right)  \right) \\
  & \;\;\; \cup \bigcup_{j=0}^{2k-1} \rot_{C_1}^{j\pi/k}
    \left(\rot_{C_1}^{\pi/kmn'_{(-1)^j}}
      \rot_{C_2}^{\pi/kmn'_{(-1)^j}}\right)
      ^{\sigma'_{(-1)^j}}
    \rot_{C_{0,0}^{\pi/4}\rot_{C_{\frac{\pi}{2},\frac{\pi}{2}}}^{\pi/4}}\Phi
    \left(\frac{1}{2kmn'_{(-1)^j}} 
    \widetilde{\mathcal{S}}_{2,m} \left(2kmn'_{(-1)^j}\frac{\pi}{4k}\right) \right).
  \end{aligned}
\end{equation}

\end{definition}

We will abbreviate the initial surfaces
$M(k,m,n_1,n_2,\sigma)$
and $N(k,m,n,n'_1,n'_{-1},\sigma_1,\sigma_{-1})$
by $M$ and $N$ respectively or sometimes indiscriminately by $\Sigma$,
when context permits.

\begin{remark}
The divisibility
assumptions are made to avoid listing a single initial surface
multiple times under different labels,
but as already acknowledged some redundancy persists in the list
in that certain items are congruent to others.
In such cases the resulting
minimal surfaces ultimately produced will also be congruent.
\end{remark}

We next collect some basic properties of the initial surfaces.

\begin{prop}
\label{initsprop}
For every choice of data, assuming 
$a>\max(R_k,R_2)$, the initial surfaces
$M(k,m,n_1,n_2,\sigma)$
and $N(k,m,n,n'_1,n'_{-1},\sigma_1',\sigma_{-1}')$
are closed, smooth surfaces embedded in $\Sph^3$. Moreover
  \begin{enumerate}[(i)]
    \item $M$ has genus $k(k-1)m(n_1+n_2)+1$;

    \item $N$ has genus $2k^2m(n'_1+n'_{-1}) + 4kmn(k-1) + 1$;

    \item $\Ckm \subset M$ and $\Grp_{k,m} \subseteq \Gsym(M)$; and
  
    \item $\Ckmp \subset N$ and $\Grp'_{k,m} \subseteq \Gsym(N)$.
  \end{enumerate}
\end{prop}

\begin{proof}
The closedness, smoothness, and embeddedness of the initial surfaces are clear from the definition and preceding discussion. For (i), the components of $\Wcal_k \backslash (C_1 \cup C_2)$ may be grouped into pairs of consecutive (in the sense of rotations about either circle) components, and the two members of each pair may then be glued to each other along $C_1$ and $C_2$ to form $k$ new (topological) tori. We get the connected sum of these tori, a surface of genus $k$, at the cost of one fundamental period of a tower along $C_1$ (or $C_2$). Each additional fundamental period, of the towers along $C_1$ and $C_2$, then contributes $k-1$ handles to the resulting surface.

The genus of $N$ is similarly calculated. Each portion of $\T'$ between two consecutive circles of intersection is glued to two portions of tori, both from alternately the $C_1$ or the $C_2$ side, orthogonally intersecting it, where these latter two are themselves glued along the circle where they intersect (so either $C_1$ or $C_2$). In this way we obtain $2k$ topological tori, whose connected sum we take at the cost of one fundamental period for all but one tower on $\T'$. Each additional period of each of these towers contributes one handle to the resulting surface, while each period of each of the remaining two towers contributes $k-1$ handles.

The great circles in the scaffoldings pass uninterrupted through the desingularized circles of intersection by virtue of the positioning and scaling in \ref{initsdef} and the fact that $\Phi$ maps the horizontal lines on the Euclidean towers to great circles. That reflections through these geodesics belong to the stabilizers of the initial surfaces follows from the intertwining \ref{phintertwine} by $\Phi$
of symmetries of $\Sk$ (\ref{tower})
with symmetries of the configurations (\ref{configsym}) to be desingularized.
\end{proof}

\begin{definition}
\label{unitnormal}
Since each initial surface is embedded in $\Sph^3$,
it is also orientable and therefore possesses
a unique global unit normal, henceforth denoted $\nu$, which points in the positive direction (meaning toward $(i,0)$) along $C_1$ at $(1,0) \in \C^2 \subset \Sph^3$.
We will write $g$ for the metric on each initial surface
induced by its defining embedding in $(\Sph^3,\gsph)$,
$A$ for its $\nu$-directed scalar-valued second fundamental form,
and $H$ for its $\nu$-directed scalar-valued mean curvature.
\end{definition}

\subsection*{Remarks on additional symmetries and the surfaces of Choe and Soret}
Before proceeding with the construction, we pause to elaborate briefly
on the symmetry groups, that is the full stabilizers in $O(4)$, of the initial surfaces,
in one particular class of highly symmetric cases. 
The groups presented in the above proposition
are the minimum symmetry groups enforced throughout the construction,
but in general each such group will be properly contained in the symmetry
group of a given initial surface (consistent with that group) as well as of the 
corresponding final minimal surface. To illustrate, consider the initial
surfaces $M(k,m,1,1,\sigma)$ of type $M$ with $n_1=n_2=1$.
Whatever the values of $k$, $m$, and $\sigma$, the full symmetry group
here will always contain reflection through not just 
$C_{\frac{j\pi}{km},\frac{(j+\ell m)\pi}{km}}$
for all $j,\ell \in \Z$ (already represented in $\Grp_{k,m}$)
but also $C_{\frac{j\pi}{2km},\frac{(j+\ell m)\pi}{2km}}$
for $j, \ell$ odd.

When $m$ is odd, $M(k,m,1,1,0)$ and $M(k,m,1,1,1)$ are equivalent under
ambient isometries
(as discussed in the Alignment subsection immediately preceding \ref{initsdef}),
so we may assume $\sigma=0$. The full symmetry
group then admits $\rot_{C'_1}^\pi$, excluded from $\Grp_{k,m}$,
and therefore also $\rot_{C'_j}^\pi$ for every odd $j$,
because (see \ref{symcomp}) for every integer $j$
the product
$\rot_{C'_{j+2}}^\pi\rot_{C'_j}^\pi
=\rot_{C_1}^{2\pi/k}\rot_{C_2}^{-2\pi/k} \in \Gsym(M)$.
(When $\sigma=1$, one has instead reflection through $C'_j$ for each even $j$,
because
$\rot_{C'_2}^\pi=
\left(\rot_{C_1}^{\pi/k}\rot_{C_2}^{\pi/k}\right)\rot_{C_2}^{2\pi/k}\rot_{C'_1}^\pi$
and by assumption $\pi/k$ is an odd multiple of $\pi/kmn$).
In this case there are no other circles of reflection on $\T'$ that are
parallel to $C_1$ and $C_2$ through tori right-handed along them.
Indeed the only such circles through which reflection preserves $\Wcal_k$
are the $C_j$ for $j \in \frac{1}{2}\Z$.
That even values of $j$ are inadmissible follows from the last
parenthetical remark.
That half-integer values are inadmissible follows also from \ref{symcomp}
since $\rot_{C'_{3/2}}^\pi\rot_{C'_1}^\pi=\rot_{C_1}^{\pi/2k}\rot_{C_2}^{\pi/2k}\rot_{C_2}^{-\pi/k}$
takes for example $(1,0) \in M$ to $(e^{-i\pi/2k},0) \not \in M$,
since $\pi/2k$ is a nonintegral multiple of the half-period $\pi/km$ when $m$ is odd.
There are, however, also circles of reflection orthogonal to the $C'_j$:
since
$C_{0,0}$ and $C'_1$ are circles of reflection of $M(k,m,1,1,0)$ (for $m$ odd still),
by \ref{symcomp} $C''_0$ is also a circle reflection,
so by \ref{symcomp} again $C''_{2j\pi/km}$ is too for every $j \in \Z$.

When $m$ is even, $M(k,m,1,1,0)$ and $M(k,m,1,1,1)$ are inequivalent.
The first ($\sigma=0$) surface has reflectional symmetry through
$C'_j$ for every integer $j$
(since $\rot_{C'_{j+1}}^\pi\rot_{C'_j}^\pi=\rot_{C_1}^{\pi/k}\rot_{C_2}^{\pi/k} \rot_{C_2}^{-2\pi/k}$,
$m$ is even, and $n=1$).
Again one also has $C''_{2j\pi/km}$ as a circle of reflection
for every $j \in \Z$.
Because $\rot_{C'_{3/2}}^\pi\rot_{C'_1}^\pi=\rot_{C_1}^{\pi/2k}\rot_{C_2}^{\pi/2k}\rot_{C_2}^{-\pi/k}$
preserves the initial surface if and only if $\pi/2k$ is an odd number of half-periods,
one has as well reflectional symmetry through $C'_j$ for every half-integer $j$
precisely when $m$ is not divisible by $4$.

The second ($\sigma=1$) surface
has symmetry group including 
$\rot_{C'_j}^{\pi}\rot_{C_1}^{\frac{\pi}{km}}\rot_{C_2}^{\frac{\pi}{km}}$
and excluding $\rot_{C'_j}^\pi$
for every integer $j$.
Thus $C''_0$ is not a circle of reflection,
but instead $C''_{j\pi/km}$ is for every $j \in 2\Z+1$.
When $m$ is divisible by $4$ the symmetry group also includes
reflection through $C'_j$ for every half-integer $j$ (but for no integer $j$),
but when $m$ is not divisible by $4$ there are no circles of reflection on $\T'$
parallel to $C_1$ and $C_2$ through tori right-handed along them.

Note that this description of the initial surfaces $M(k,m,1,1,\sigma)$
with $m$ even is consistent
with the properties of the surfaces constructed by Choe and Soret in \cite{CS};
specifically the genus and symmetries of the $\sigma=0$ and $\sigma=1$ surfaces match those 
of, respectively, the \emph{odd} and \emph{even} surfaces in \cite{CS}.

\begin{remark}
\label{r:unique} 
Note that to prove that the surfaces we construct in \ref{mainthm} 
of type $M$ with $m$ even and $n_1=n_2=1$ are the same
surfaces found in \cite{CS}, 
it is enough to prove the uniqueness of the solutions to the Plateau problems in \cite{CS}. 
Although this seems very likely to be true, 
we do not have a proof at the moment. 
Note also that in \cite{Lawson} Lawson claims uniqueness for the solution to his Plateau problem. 
\end{remark}

\section{The extended standard regions}
Every initial surface is covered by certain open sets,
which we call \emph{extended standard regions}, of two types.
Regions of the first type are indexed by the circles of intersection.
Given a particular initial surface $\Sigma$,
for any circle $C$ of intersection in the corresponding initial configuration we define
$\mC$
(suppressing dependence on the given initial surface)
to be the number of fundamental periods of the tower in $\Sigma$ wrapped around $C$,
so that (recall \ref{initsdef})
$\mC$ is the product of $m$
with the appropriate factor involving $k$ or $2$
as well as $n$, $n_1$, or $n_2$, depending on $\Sigma$ and $C$:
  \begin{equation}
    \mC
    :=
    \begin{cases}
      kmn_j \mbox{ if $\Sigma=M(k,m,n_1,n_2,\sigma)$ and $C=C_j$} \\
      2kmn \mbox{ if $\Sigma=N(k,m,n,n'_1,n'_{-1},\sigma'_1,\sigma'_{-1})$ and $C=C_1$ or $C=C_2$} \\
      2kmn'_{(-1)^j} \mbox{ if $\Sigma=N(k,m,n,n'_1,n'_{-1},\sigma'_1,\sigma'_{-1})$
                              and $C=C'_j$.}
    \end{cases}
  \end{equation}
We similarly define
  \begin{equation}
    \kC
    :=
    \begin{cases}
      k \mbox{ if $C=C_1$ or $C=C_2$} \\
      2 \mbox{ otherwise}.
    \end{cases}
  \end{equation}

Then, recalling \ref{adef}, we let
\begin{equation}
S[C]:=\left\{p \in \Sigma: d_{\gsph}(p,C) < \frac{a}{\mC}\right\},
\end{equation}
where $d_{\gsph}(p,C)$ denotes the distance in $\Sph^3$ between $p$ and $C$.
This region is naturally identified
with a truncated Karcher-Scherk tower
$\SkC(a):=\SkC \cap \{\sqrt{\xx^2+\yy^2}<a\}$ via the map
\begin{align}
&\phi_{_{C,m}}:  \R^3 \to \Sph^3, \text{ defined by} \\
&\phi_{_{C,m}}(\xx,\yy,\zz) := \rot[C] \Phi \left(\frac{(\xx,\yy,\zz)}{\mC}\right),
\end{align}
where $\rot[C]$ is an element of $SO(4)$ chosen so that 
$S[C] = \varphi_{C,m} \left( \SkC(a) \right)$.
Sometimes we may suppress the dependence on $m$, writing simply $\phiC$.
In turn we define $\xC: \SkC(a) \to S[C]$ by
  \begin{equation}
    \xC
    :=
    \left. \phiC \circ \iota_{_{\SkC}}\right|_{\SkC(a)}
    =\left. \phiC \circ \xkCm \right|_{\SkC(a)}.
  \end{equation}

A new constant $b>\max(R_2,R_k)$, to be determined later, dictates the extent of the second type of region. It will be chosen independently of $m$ but large enough so that each such region closely approximates a Clifford torus.
Regions of the second type in an initial surface $\Sigma$
are indexed by the connected components
of the complement of the circles of intersection
in the initial configuration that $\Sigma$ desingularizes.
Given such a component $T$, with boundary $\partial T = C \cup D$,
we may assume---if necessary
by redefining $\phi_C$ by precomposition with a symmetry of $\SkC$
and $\phi_D$ by precomposition with a symmetry of $\mathcal{S}_{k_{_D}}$---that
$T$ has inward unit conormals $\mC d\phiC \partial_{\xx}$
and $m_{_D}d\phi_{_D} \partial_{\xx}$ along $C$ and $D$.
Then, recalling \ref{tower} and for any integer $j \geq 2$ setting
\begin{equation}
  H_j := \{ (\xx,\cutoff{a+1}{a}W_j(\xx,\zz),\zz) 
    \in \R^3 : \xx \in (b,a+1], \, \zz \in \R \},
\end{equation}
we define
\begin{equation}
\label{STdef}
  S[T] := \phiC \left(H_{\kC}\right) \cup 
    \phi_{_D}\left(H_{k_{_D}}\right)
    \cup (\Sigma \cap T \cap \{d_{\gsph}(\cdot,C)>b/\mC\} \cap d_{\gsph}(\cdot,D)>b/m_{_D}\}).
\end{equation}
This region is naturally identified with
\begin{equation}
T_b
:=
\left\{p \in T : d_{\gsph}(p,C)>\frac{b}{\mC}
\text{ and } d_{\gsph}(p,D)>\frac{b}{m_{_D}}\right\}
\end{equation}
via the map
\begin{align}
&\varpi_{_T}=\varpi_{_{T,m}}: S[T] \to T_b, \text{ defined by} \\
&\varpi_{_{T,m}}(p) = \begin{cases}
\phiC \circ \pi_{\xx\zz} \circ \phiC^{-1}(p)
\text{ if } p \in \phiC \left(H_{\kC}\right)  \\
\phi_{_D} \circ \pi_{\xx\zz} \circ \phi_{_D}^{-1}(p)
\text{ if } p \in \phi_{_D} \left(H_{k_{_D}}\right)  \\
p \text{ otherwise},
\end{cases}
\end{align}
where $\pi_{\xx\zz}: \R^3 \to \R^3$ is Euclidean orthogonal projection onto the $\xx\zz$-plane. It is immediate that $\varpi_{_{T,m}}$ is well-defined and smooth.

Given an initial surface $\Sigma$, we write $\intcirc(\Sigma)$ for the collection of circles of intersection in the corresponding initial configuration and 
$\comptor(\Sigma)$ for the collection of components of the complement of 
$\bigcup_{C \in \intcirc(\Sigma)} C$ in the initial configuration. Then 
$\Sigma =\bigcup_{C \in \intcirc(\Sigma)} S[C] 
  \cup \bigcup_{T \in \comptor(\Sigma)} S[T]$,
the members of $\{S[C]: C \in \intcirc(X)\}$ are pairwise disjoint, the members of 
$\{S[T]: T \in \comptor(\Sigma)\}$ are pairwise disjoint, and 
$S[C] \cap S[T] = \emptyset$ unless $C \subset \partial T$.

\begin{figure}
    \includegraphics[width=7in]{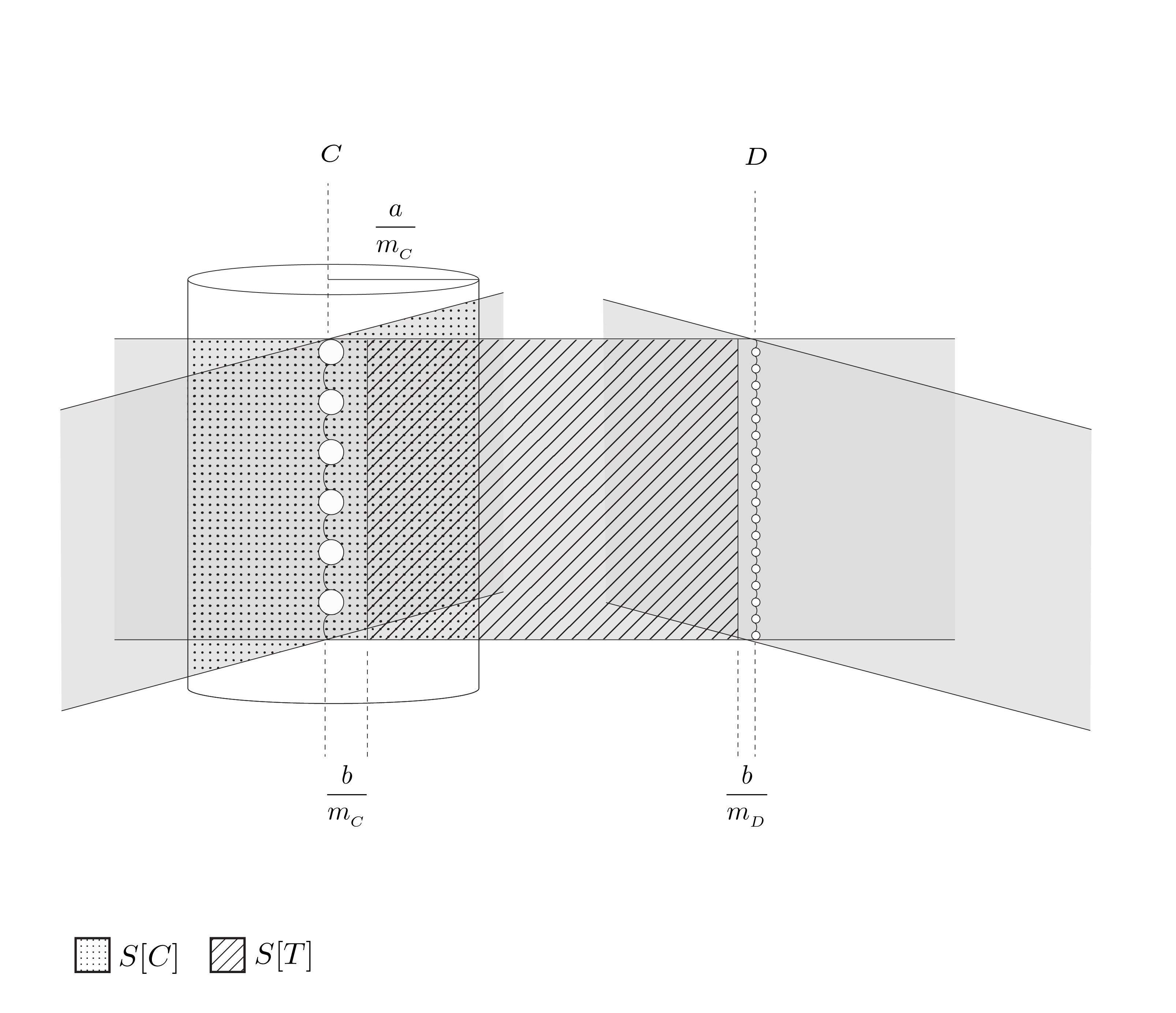}
      \caption{Extended standard regions}
  \end{figure}

Using the diffeomorphisms just defined, the next two propositions compare
the extended standard regions, as embeddings in $(\Sph^3,m^2\gsph)$,
to standard Karcher-Scherk towers and planes in Euclidean space.

\begin{prop}
\label{stdtor}
Let $\Sigma$ be an initial surface
and $T \in \comptor(\Sigma)$ a toral component.
Write $g_{_T}$ for the flat metric on $T$,
and for each point $p \in T_b$
let $d_{\gsph}(\partial T_b,p)$
denote the distance from $p$ to the boundary circles of $T_b$.
Then for every nonnegative integer $\ell$ there exists a constant
$C(\ell)$---independent of $m$---such that
  \begin{enumerate}[(i)]
    \item $\norm{m^2(g-\varpi_T^*g_{_T}): C^\ell \left(
      T^*S[T]^{\otimes 2},\; m^2\varpi_T^*g_{_T},\; 
      e^{-m\varpi_T^*d_{\gsph}(\partial T_b,\cdot)}
      \right)} \leq C(\ell)$;

    \item $\norm{m^{-2}(\abs{A}^2-2): C^\ell \left(
      S[T],\; m^2\varpi_T^*g_{_T},\;
      e^{-m\varpi_T^*d_{\gsph}(\partial T_b,\cdot)}
      \right)} \leq C(\ell)m^{-1}$; and

    \item $\norm{m^{-2}H: C^\ell \left(
      S[T],\; m^2\varpi_T^*g_{_T},\;
      e^{-m\varpi_T^*d_{\gsph}(\partial T_b,\cdot)}
      \right)} \leq C(\ell)m^{-1}$.

  \end{enumerate}
\end{prop}

\begin{proof}
We select a boundary circle $C$ of $T$ and a rotation $\rot[C] \in SO(4)$
so that $T \subset \rot[C] \Phi(\{\yy=0,\;x > 0\})$. It suffices to establish
the estimates within 
$\mathcal{U}:=\{p \in \Sph^3 \; : \; d_{\Sph^3}(p,C) < \frac{a+1}{\mC}\}$.
Using \ref{phullback} we find
  \begin{equation}
  \label{gsph}
    \begin{aligned}
      \Phi^*\rot[C]^*\gsph = &d\xx^2 + d\yy^2 + d\zz^2
      + \frac{\sin^2 \rr - \rr^2}{\rr^4} 
      \left(
        \yy^2\, d\xx^2 + \xx^2\,d\yy^2 - 2\xx\yy\,d\xx\,d\yy
      \right) \\
      &+ 2\frac{\sin^2 \rr}{\rr^2}
      \left(
        \xx\,d\yy\,d\zz - \yy\,d\xx\,d\zz
      \right),
    \end{aligned}
  \end{equation}
whose components and whose inverse's components have (coordinate) derivatives of
all orders bounded on
$\Phi^{-1}\rot[C]^{-1}\mathcal{U}
\subset \{\sqrt{\xx^2+\yy^2}<\pi/4\} \subset \R^3$.
Moreover $\varpi_T$ identifies
$S[T]$ as a graph over $T$ so that
  \begin{equation}
    \begin{aligned}
      &S[T] \cap \mathcal{U} =
        \rot[C]\Phi\{(\xx,f(\xx,\zz),\zz) \; : \; 
        \rot[C]\Phi(\xx,0,\zz) \in \varpi_T^{-1}\left(S[T] \cap \mathcal{U}\right)\},
        \mbox{ where} \\
      &f(\xx,\zz) = \mC^{-1}W_{\kC}(\mC\xx,\mC\zz)
        \cutoff{a+1}{a}(\mC\xx),
    \end{aligned}
  \end{equation}
recalling $W_{\kC}$ from \ref{tower}. From \ref{towerdecay} we have for any nonnegative
integers $j$ and $\ell$ the existence of a constant $C(j,\ell)$ ensuring the estimate
  \begin{equation}
  \label{fest}
    \abs{\partial_{\xx}^j\partial_{\zz}^{\ell}f(\xx,\zz)}
    \leq C(j,\ell)m^{j+\ell-1}e^{-m\xx}
  \end{equation}
for any $(\xx,\zz) \in \left[\frac{b}{\mC},\frac{a+1}{\mC}\right]$.

Now, via $\rot[C]\Phi$,
the coordinates $(\xx,\yy,\zz)$ on $\R^3$ transfer to
$\mathcal{U} \subset \Sph^3$
and the functions $(\xx,\zz)$ restrict to coordinates on $T$, so
that $g_{_T} = d\xx^2 + d\zz^2$ and
  \begin{equation}
    \begin{aligned}
      \left[\left(\varpi_T^{-1}\right)^*g\right]_{ij}(\xx,\zz) =
      &\left[\gsph\right]_{ij}(\xx,f(\xx,\zz),\zz)
      + f_{,i}(\xx,\zz)\left[\gsph\right]_{i\yy}(\xx,f(\xx,\zz),\zz) \\
      &+ f_{,j}(\xx,\zz)\left[\gsph\right]_{j\yy}(\xx,f(\xx,\zz),\zz)
      + f_{,i}(\xx,\zz)f_{,j}(\xx,\zz)\left[\gsph\right]_{\yy\yy}(\xx,f(\xx,\zz),\zz),
    \end{aligned}
  \end{equation}
whence follows the estimate (i) for the metric, in light of \ref{fest} 
and the boundedness of all (coordinate) derivatives of all components of $\gsph$
and its inverse with respect to the $(\xx,\yy,\zz)$ coordinate system as established
in \ref{gsph}.

Assuming the normal $\nu$ on $S[T]$ has positive inner product with $\partial_{\yy}$,
we calculate also
  \begin{equation}
    \begin{aligned}
      A_{ij}(\xx,\zz) = 
      &\left(
      \left[\gsph\right]^{\yy \yy}(\xx,f(\xx,\zz),\zz)
      - f_{,k}(\xx,\zz)\left[\gsph\right]^{k \yy}(\xx,f(\xx,\zz),\zz)
      + f_{,k}(\xx,\zz)f_{,\ell}(\xx,\zz)\left[\gsph\right]^{k \ell}(\xx,f(\xx,\zz),\zz)
      \right)^{-\frac{1}{2}} \cdot \\
      &[
      \Gamma_{ij}^{\yy}(\xx,f(\xx,\zz),\zz)
      + f_{,ij}(\xx,\zz)
      + f_{,j}(\xx,\zz)\Gamma_{i\yy}^{\yy}(\xx,f(\xx,\zz),\zz) \\
      &+ f_{,i}(\xx,\zz)\Gamma_{j\yy}^{\yy}(\xx,f(\xx,\zz),\zz)
      + f_{,i}(\xx,\zz)f_{,j}(\xx,\zz)\Gamma_{\yy \yy}^{\yy}(\xx,f(\xx,\zz),\zz) \\
      &- f_{,k}(\xx,\zz)\Gamma_{ij}^k(\xx,f(\xx,\zz),\zz)
      - f_{,k}(\xx,\zz)f_{,j}(\xx,\zz)\Gamma_{i\yy}^k(\xx,f(\xx,\zz),\zz) \\
      &- f_{,k}(\xx,\zz)f_{,i}(\xx,\zz)\Gamma_{j\yy}^k(\xx,f(\xx,\zz),\zz)
      - f_{,k}(\xx,\zz)f_{,i}(\xx,\zz)f_{,j}(\xx,\zz)
        \Gamma_{\yy \yy}^k(\xx,f(\xx,\zz),\zz)
      ],
    \end{aligned}
  \end{equation}
where $i,j,k,\ell \in \{\xx,\zz\}$ and each instance of $\Gamma$ is a Christoffel
symbol of $\gsph$ in the $(\xx,\yy,\zz)$ coordinate system. Noting that the squared
norm of the second fundamental form of $T$ is simply $2$ and that $T$ is minimal, 
we obtain (ii) and (iii), using again
the estimates \ref{fest} and the boundedness exhibited by \ref{gsph}.
\end{proof}

\begin{prop}
\label{stdtow}
Let $\Sigma$ be an initial surface
and $C \in \intcirc(\Sigma)$ an intersection circle.
Write $\Ahat_{_C}$ for the second fundamental form
of the inclusion $\iota_{_{\SkC}}: \SkC \to \R^3$
relative to the rescaled Euclidean metric 
$\frac{m^2}{\mC^2}\geuc$ and
the unit normal whose pushforward by $\phi_{_{C,m}}$ has positive $\gsph$ inner product
with $\nu$, and set $\ghat_{_C}=\frac{m^2}{\mC^2}\iota_{_{\SkC}}^*\geuc$.
Then for each nonnegative integer $\ell$ there exists a constant 
$C(\ell)$---independent of $m$---such that
  \begin{enumerate}[(i)]
     \item $\norm{m^2g - {\xC^{-1}}^*\ghat_{_C} :
      C^{\ell}\left(T^*S[C]^{\otimes 2},\; m^2g\right)}
      \leq C(\ell)m^{-1/2}$;

    \item $\norm{m^{-2}\abs{A}^2
      -{\xC^{-1}}^*\abs{\Ahat_{_C}}^2: 
      C^\ell(S[C], \; m^2g)}
      \leq C(\ell)m^{-1/2}$; and

    \item $\norm{m^{-2}H:
      C^\ell(\Sigma, \; m^2g)}
      \leq C(\ell)m^{-3/2}$.
  \end{enumerate}
\end{prop}

\begin{proof}
From \ref{gsph}
  \begin{equation}
    \norm{
    \phi_{_C}^* \mC^2\gsph - \geuc :
    C^\ell\left(\{\sqrt{\xx^2+\yy^2} \leq a \} 
    \subset \R^3, \, \geuc, \,
    \frac{\sqrt{\xx^2+\yy^2}+1}{\mC} \right)
    }
    \leq C(\ell).
  \end{equation}
Since the second fundamental form of $\Sk$ is bounded, as is each of its
covariant derivatives, we obtain
  \begin{equation}
    \begin{aligned}
      &\norm{
        m^2g - {\xC^{-1}}^*\ghat_{_C}:
        C^\ell(S[C], \, m^2g, \, m^{-1}+d_{\gsph}(C,\cdot))
        } \leq C(\ell), \\
      &\norm{
        m^{-2}\abs{A}^2 - {\xC^{-1}}^*\abs{\Ahat_{_C}}^2:
        C^\ell(S[C], \, m^2g, \, m^{-1}+d_{\gsph}(C,\cdot))
      } \leq C(\ell), \mbox{ and}  \\
      &\norm{
        m^{-1}H : C^\ell(S[C], \, m^2g, \, m^{-1}+d_{\gsph}(C,\cdot))
        } \leq C(\ell),
    \end{aligned}
  \end{equation}
establishing the estimates of the proposition on the subset of $S[C]$
within the tubular neighborhood of center $C$ and $\gsph$-radius $m^{-1/2}$.
By assuming $\sqrt{m}>b$, we ensure that the complement of this subset
falls under the regime of the preceding proposition,
which in conjunction with the asymptotic geometry of $\SkC$ itself
completes the proof,
under the further assumption that $e^{\sqrt{m}}>\sqrt{m}$.
\end{proof}

\section{The linearized equation}
\label{linear}

Given an initial surface $\Sigma$, embedded in $\Sph^3$ by 
$X: \Sigma \to \Sph^3$, along with 
a function $u \in C^2(\Sigma)$ with sufficiently small $C^0$ norm,
and recalling (\ref{unitnormal}) the choice $\nu: \Sigma \to T\Sph^3$ of global unit normal,
we define the immersion
  \begin{equation}
    \begin{aligned}
      &X_u: \Sigma \to \Sph^3 \\
      &p \mapsto \exp_{_{X(p)}} u(p)\nu(p),
    \end{aligned}
  \end{equation}
where $\exp: T\Sph^3 \to \Sph^3$ is the exponential map on
$(\Sph^3,\gsph)$. 
Write $\nu_u$ for the global unit normal on $X_u$
which has nonnegative inner product with the velocity field for the geodesics generated by $\nu$
and write $\Hcal[u]$ for the scalar mean curvature of $X_u$ relative to $\nu_u$.
Write $\Grp$ for (recalling \ref{symm}) $\Grp_{k,m}$ when $\Sigma$ is type $M$
and for $\Grp'_{k,m}$ when $\Sigma$ is type $N$.
The main theorem will be proven by selecting
a solution
$u \in C_{\Grp}^\infty(\Sigma)$
(the space of smooth $\Grp$-odd functions on $\Sigma$---recall \ref{invfun})
to $\Hcal[u]=0$, small enough that $X_u$ is an embedding.

To that end we next study the linearization $\Lcal$ at $0$ of $\Hcal$, given by
  \begin{equation}
    \Lcal u = \left.\frac{d}{dt}\right|_{t=0}\Hcal[tu] 
    = \left(\Delta + \abs{A}^2 + 2\right)u,
  \end{equation}
where the constant term $2$ arises as the Ricci curvature of $\gsph$ contracted twice with $\nu$.
Actually, to secure bounds uniform in $m$ and to facilitate the comparison of this operator
on the extended standard regions to certain limit operators, we focus on
  \begin{equation}
    m^{-2}\Lcal: C^{2,\beta}_{\Grp}(\Sigma,m^2g)
    \to C^{0,\beta}_{\Grp}(\Sigma,m^2g).
  \end{equation}
In view of the estimates of the second fundamental form contained in \ref{stdtor}
and \ref{stdtow}, this operator is bounded for any $\beta \in (0,1)$.
The present section is
devoted to obtaining a bounded inverse by first analyzing the operator
``semilocally''---meaning when restricted to spaces of functions defined on each of the various extended standard regions---and by afterward applying an iteration scheme to piece together a global solution.

\subsection*{Approximate solutions on towers}
We first solve the Jacobi equation for the inclusion map
$\iota_{_{\Sk}}: \Sk \to \R^3$
of the exact Karcher-Scherk towers of standard size,
given data with sufficiently small support.
To avoid the introduction of substitute kernel needed in more complicated gluing
constructions, we impose the symmetries $\Grp$ induces on the limit tower.
Specifically, set $\ghat=\iota_{_{\Sk}}^*\geuc$, write $\Ahat$ for the second fundamental
form of $\iota_{_{\Sk}}: \Sk \to \R^3$, and for each positive integer $n$ let
  \begin{equation}
    \Grphat:=\Grphat_{k,n}:=\Grefl(n\Ck)
    =
    \left\langle
      \rote_{\xx \text{-axis}}^\pi, \, \rote_{\zz \text{-axis}}^{2\pi/k}, \,
      \transe_{\zz \text{-axis}}^{2n\pi}
    \right\rangle
    \subsetneq \Gsym(\Sk).
  \end{equation} 
Then the Jacobi operator
  \begin{equation}
    \Lhat = \Delta_{\ghat} + \abs{\Ahat}^2
  \end{equation}
defines a bounded linear map 
$\Lhat: C_{\Grphat}^{2,\beta}\left(\Sk,\ghat\right)
  \to C_{\Grphat}^{0,\beta}\left(\Sk,\ghat\right)$
for any $\beta \in (0,1)$ and integers $k \geq 2$ and $n \geq 1$.
Given $b>0$ we recall that $\Sk(b)=\Sk \cap \{\sqrt{\xx^2+\yy^2}<b\}$
and we set
  \begin{equation}
    C_{c,\; \Grphat}^{0,\beta}\left(\Sk(b),\ghat \right)
    := \left\{ u \in C_{\Grphat}^{0,\beta}\left(\Sk,\ghat \right) \mbox{
       having support compactly contained in } \Sk(b) \right\}.
  \end{equation}

\begin{prop}
\label{towsol}
Fix $\beta \in (0,1)$, $b>0$, and integers $k \geq 2$ and $n \geq 1$. 
Then there exists a linear map
  \begin{equation}
    \Rtowkn : C_{c, \; \Grphat}^{0,\beta}\left(\Sk(b),\ghat\right)
    \to
    C_{\Grphat}^{2,\beta}\left(\Sk,\ghat\right)
  \end{equation}
and there exists a constant $C$---depending on just $\beta$, $b$, $k$, and
$n$---such that for any $f$ in the domain of $\Rtowkn$ we have
$\Lhat \Rtowkn f = f$ and
  \begin{equation}
    \norm{\Rtowkn f : C^{2,\beta}\left(\Sk,\ghat\right)}
    \leq
    C \norm{f: C^{0,\beta}\left(\Sk,\ghat \right)}.
  \end{equation}
\end{prop}

\begin{proof}
The proof will be completed by (i) introducing a conformal metric $\eta=e^{2\phi}\ghat$,
(ii) establishing that the Schr\"{o}dinger operator
$\Lhat_\eta:=e^{-2\phi}\Lhat$ acting between Sobolev spaces
defined with respect to this metric has discrete spectrum omitting $0$,
(iii) extracting a $C^0$ bound for solutions, and finally (iv) applying Schauder estimates.
For $k>2$ the tower $\Sk$ has umbilic points where the Gauss map takes vertical values, 
so the pullback $h=\frac{1}{2}\abs{\Ahat}^2\ghat$ of the round
spherical metric $\gsph$ by the Gauss map,
as applied in \cite{KapJDG} and other constructions, degenerates there.

Instead we
will pull back the spherical metric by a different map.
Recall that above we studied the Enneper-Weierstrass representation \ref{WeiMap}
on the unit disc, which, after a similarity transformation,
parametrized a half-period of $\Sk$.
In fact we observe that \ref{WeiMap} extends to a diffeomorphism
  \begin{equation}
      \xi: (\C \cup \{\infty\}) \backslash \{\omega_1, \cdots, \omega_{2k}\}
      \to
      \frac{1}{k}\rote_{\zz\text{-axis}}^{\pi/2k}\Sk /
        \left\langle \transe_{\zz\text{-axis}}^{2\pi/k} \right\rangle
  \end{equation}
from the extended complex plane $\C \cup \{\infty\}$ punctured at the $2k$ roots of $-1$
to the corresponding tower modulo vertical translation by $2\pi/k$.
Moreover, the inverse extends to a covering map
  \begin{equation}
    \xi^{-1}: \frac{1}{k}\rote_{\zz\text{-axis}}^{-\pi/2k}\Sk
    \to
    (\C \cup \{\infty\}) \backslash \{\omega_1, \cdots, \omega_{2k}\}
  \end{equation}
of the punctured extended plane by the full tower.
By composing with the appropriate rotation, scaling, and
the inverse of stereographic projection $\varpi: \Sph^2 \to \C \cup \{\infty\}$,
we obtain a smooth covering
  \begin{equation}
    \Pi:=\varpi^{-1} \circ \xi^{-1} \circ \frac{1}{k} \circ \rote_{\zz\text{-axis}}^{-\pi/2k}:
      \Sk \to \Sph^2 \backslash \bigcup_{j=1}^{2k} \{\omega_j\}.
  \end{equation} 
(The Gauss map $\hat{\nu}: \Sk \to \Sph^2$ is then just
$\hat{\nu}(p)
=\rote_{\zz\text{-axis}}^{\pi/2k}\varpi^{-1}\left(\varpi \circ \Pi(p)\right)^{k-1}$.)

Referring further to the Enneper-Weierstrass data \ref{WeiData} we deduce
(see for example \cite{KarT} or any standard reference for the classical theory of minimal surfaces)
  \begin{equation}
    \xi^*\ghat(z)=\left(\frac{\abs{z}^{2k-2}+1}{\abs{z^{2k}+1}}\right)^2\abs{dz}^2,
  \end{equation}
so since
  \begin{equation}
    {\varpi^{-1}}^*\gsph(z)=\frac{4}{\left(\abs{z}^2+1\right)^2}\abs{dz}^2,
  \end{equation}
we find
  \begin{equation}
  \label{eta}
    \eta:=\Pi^*\gsph = 
      \frac{4\abs{(\varpi \circ \Pi)^{2k} + 1}^2}
      {k^2\left(\abs{\varpi \circ \Pi}^2 + 1\right)^2
      \left(\abs{\varpi \circ \Pi}^{2k-2} + 1\right)^2} \ghat=:e^{2\phi}\ghat,
  \end{equation}
so that the conformal factor $e^{2\phi}$
in front of $\ghat$ in \ref{eta}
and its reciprocal $e^{-2\phi}$ are bounded
on the inverse image under $\Pi$ of every compact subset of
$\Sph^2 \backslash \bigcup_{j=1}^{2k} \{\omega_j\}$.
Also from \ref{WeiData} we find
  \begin{equation}
    \left(\frac{1}{k} \circ \rote_{\zz\text{-axis}}^{-\pi/2k} \circ \xi\right)^*\abs{\Ahat}^2(z)
    =
    8\left(\frac{k-1}{k}\right)^2\frac{\abs{z}^{2k-4}\abs{z^{2k}+1}^2}{\left(\abs{z}^{2k-2}+1\right)^4},
  \end{equation}
whence
  \begin{equation}
  \label{potential}
    e^{-2\phi}\abs{\Ahat}^2
    =
    2(k-1)^2\abs{\varpi \circ \Pi}^{2k-4}
      \left(
        \frac{\abs{\varpi \circ \Pi}^2+1}{\abs{\varpi \circ \Pi}^{2k-2}+1}
      \right)^2,
  \end{equation}
so that the potential term of $\Lhat_\eta:=e^{-2\phi}\Lhat$ is a smooth function with
absolute value bounded on all $\Sk$.
Note that pullback $(\Pi^{-1} \circ \omega^{-1})^*e^{2\phi}$
of the conformal factor in \ref{eta}
is even under reflection through every line
through opposite $2k$\textsuperscript{th} roots of $-1$;
therefore (recall the discussion of the symmetries in the proof of \ref{tower})
$e^{2\phi}$ itself is in particular $\Grphat$-even,
and so $\Lhat_\eta$, like $\Lhat$,
takes $\Grphat$-odd functions to $\Grphat$-odd functions.

For each nonnegative integer $\ell$
write $H_{\Grphat}^\ell(\Sk,\eta)$
for the Sobolev space consisting of all $\Grphat$-odd
(in the distributional sense)
measurable functions whose weak
covariant derivatives up to order $\ell$,
with respect to $\eta$, have squared norms with finite integrals
on the quotient $(\Sk/\Grphat,\eta)$; define 
the $H_{\Grphat}^\ell(\Sk,\eta)$
norm of such a function to be the square root of the sum, 
from order $0$ to order $\ell$,
of these integrals.

Although $\Sk$ is not compact
(nor is the quotient $\Sk/\Grphat$),
$(\Sk,\eta)$ is nevertheless a union of closed round hemispheres punctured on their equators:
  \begin{equation}
    \Sk = \bigcup_{j \in \Z} \overline{\Omega_j},
  \end{equation}
where the overline indicates topological closure in $\Sk$ and
  \begin{equation}
  \label{Omegadef}
    \Omega_j:=\Sk \cap \left\{ \frac{2j-1}{2}\pi < \zz < \frac{2j+1}{2}\pi \right\}
  \end{equation}
is the open region on $\Sk$ between two consecutive horizontal planes of symmetry;
each $(\Omega_j,\eta)$ is isometric to an open round hemisphere of radius $1$.

Thus, from a sequence bounded in $H^1_{\Grphat}(\Sk,\eta)$, by
applying the Rellich-Kondrashov lemma successively to $2n$ contiguous such
hemispheres and bearing in mind the $\Grphat$ equivariance, we can
extract a subsequence converging in $L^2_{\Grphat}(\Sk,\eta)=H^0_{\Grphat}(\Sk,\eta)$.
Using also the boundedness of $e^{-2\phi}\abs{\Ahat}^2$
and a standard application of
the Riesz representation theorem for Hilbert spaces,
we conclude that $\Lhat_\eta$ has discrete spectrum.
In Proposition \ref{nokernel} below we show moreover that $\Lhat_\eta$ has trivial kernel.

Now, given $f \in C^{0,\beta}_{c, \; \Grphat}(\Sk(b),\eta)$,
we have $e^{-2\phi}f \in H^0_{\Grphat}(\Sk,\eta)$ with
$\norm{e^{-2\phi}f}_{L^2(\Sk,\eta)} \leq C(b)\norm{f}_{C^{0,\beta}(\Sk,\eta)}$,
so by Proposition \ref{nokernel} there exists $u \in H^1_{\Grphat}(\Sk,\eta)$ weakly solving
$\Lhat u = f$ and satisfying the estimate 
$\norm{u}_{H^1(\Sk,\eta)} \leq C(b)\norm{f}_{C^{0,\beta}(\Sk,\eta)}$.
By standard elliptic regularity theory and the bounded geometry
of $(\Sk,\ghat)$ in fact $u \in C^{2,\beta}_{loc}(\Sk)$ with
  \begin{equation}
  \label{Schauder}
    \norm{u}_{C^{2,\beta}(B_1(p),\ghat)} \leq C \left(\norm{u}_{C^0(B_2(p))}
     + \norm{f}_{C^{0,\beta}(\Sk,\ghat)}\right)
  \end{equation}
for each $p \in \Sk$, where $B_r(p)$ is the open ball with
center $p$ and $\ghat$ radius $r$.

Next, from the Bochner formula
together with the divergence theorem, the recognition
that the compactly supported smooth functions are dense in $H^1_{\Grphat}(\Sk,\eta)$,
and the equality
$\Delta u = e^{-2\phi}f - e^{-2\phi}\abs{\Ahat}^2u$,
one secures the further estimate
  \begin{equation}
  \label{H2est}
    \norm{u}_{H^2(\Sk,\eta)} \leq C\norm{f}_{C^{0,\beta}(\Sk,\eta)}.
  \end{equation}
Now, given $v \in C^\infty(\Sk)$ and $p \in \Sk$,
there is a sector $S \subset \Sk$ of a spherical cap
(relative to $\eta$),
with center $p$, vertex angle $\pi/4$, and radius $1/4$,
entirely contained in $\Sk$
(so missing the roots of unity at the equator).
Then, writing $\gamma_\theta(s)$
for the $\eta$ geodesic through $p$,
parametrized by arc length $s$ from $p$,
and with initial angle $\theta$
measured from one edge of $S$ to the other,
we have
$v(p)
=-\int_0^{1/4} \frac{d}{ds}(\cutoff{1/4}{1/8}(\gamma_\theta(s))v(\gamma_\theta(s))) \, ds
=\int_0^{1/4} s\frac{d^2}{ds^2}(\cutoff{1/4}{1/8}(\gamma_\theta(s))v(\gamma_\theta(s))) \, ds,
$
using the fundamental theorem of calculus and integrating by parts.
Integrating in $\theta$ from $0$ to $\pi/4$ and applying the Cauchy-Schwarz
inequality yields the
simple Morrey-Sobolev inequality
  \begin{equation}
    \norm{u}_{C^0(\Sk)} \leq C\norm{u}_{H^2(\Sk,\eta)}.
  \end{equation}
This estimate, in conjunction with \ref{H2est} and \ref{Schauder},
completes the proof.
\end{proof}

\begin{prop}
\label{nokernel}
The operator $\Lhat_\eta$ acting on $H_{\Grphat}^1(\Sk,\eta)$,  
as defined in the proof of Proposition \ref{towsol}, has trivial kernel.
\end{prop}

\begin{proof}
To show that $\Lhat_\eta$ has trivial $H_{\Grphat}^1(\Sk,\eta)$ kernel
we will first count its nullity on the somewhat larger domain $H_{\Hhat}^1(\Sk,\eta)$,
where
  \begin{equation}
    \Hhat
    :=
    \Hhat_{k,n}
    :=
    \left\langle
      \rote_{\zz \text{-axis}}^{2\pi/k}, \, \transe_{\zz \text{-axis}}^{2n\pi}
    \right\rangle
    \subsetneq
    \Grphat
    \subsetneq
    \Gsym(\Sk)
  \end{equation}
is the subgroup of $\Grphat$ having the same generators save the
reflections through lines, which are excluded.
This count is performed by adapting the variational
proofs given by Montiel and Ros for Lemma 12 and Lemma 13 in \cite{MR}.
There they calculate the multiplicity of the eigenvalues of the Laplacian
on the round sphere as eigenvalues of the Laplacian on certain coverings
of the sphere. Here we are interested in the multiplicity of $0$ only, for
the more complicated operator $\Lhat_\eta$.

Recalling \ref{Omegadef} we see that
$\Hhat \left(\bigcup_{j=1}^{2n} \overline{\Omega_j}\right) = \Sk$.
Write $H^1_{\Hhat}(\Omega_j,\eta)$ for the space of restrictions to
$\Omega_j$ of elements of $H^1_{\Hhat}(\Sk,\eta)$ and
write $H^1_{0,\Hhat}(\Omega_j,\eta)$ for the closure in $H^1_{\Hhat}(\Sk,\eta)$
of the space of smooth functions compactly supported in $\Omega_j$ and invariant under
the (rotational) symmetries of $\Hhat$ that preserve $\Omega_j$.
We define the bilinear form $B_{\Lhat_\eta,\Omega_j}$
on $H^1_{\Hhat}(\Omega_j) \times H^1_{\Hhat}(\Omega_j)$
  \begin{equation}
    B_{\Lhat_\eta,\Omega_j}[u,v]
    :=
    -(du,dv)_{L^2(\eta)}+\left(e^{-2\phi}\abs{\Ahat}^2u,v\right)_{L^2(\eta)}.
  \end{equation}
Then mimicking \cite{MR}, we define the subspace $V$ of $H^1_{\Hhat}(\Sk,\eta)$ by
  \begin{equation}
    \begin{aligned}
      V:= &\left\langle \left. u \in H^1_{0,\Hhat}(\Omega_1,\eta) \; \right| \;
          \exists \lambda>0 \;\; \forall v \in H^1_{0,\Hhat}(\Omega_1,\eta) \;\;
          B_{\Lhat_\eta,\Omega_1}[u,v]
          = \lambda (u,v)_{L^2(\eta)}
          \right\rangle \\
          &\,\,\, \oplus \,
          \bigoplus_{j=2}^{2n} \left\langle \left. u \in H^1_{0, \Hhat}(\Omega_j,\eta) \; \right| \;
          \exists \lambda \geq 0 \;\; \forall v \in H^1_{0,\Hhat}(\Omega_j,\eta) \;\;
          B_{\Lhat_\eta,\Omega_j}[u,v]
          = \lambda (u,v)_{L^2(\eta)}
          \right\rangle,
    \end{aligned}
  \end{equation}
and the vector space $W$ by
  \begin{equation}
    \begin{aligned}
    W:= &\left\langle \left. u \in H^1_{\Hhat}(\Omega_1,\eta) \; \right| \;
          \exists \lambda \geq 0 \;\; \forall v \in H^1_{\Hhat}(\Omega_1,\eta) \;\;
          B_{\Lhat_\eta,\Omega_1}[u,v]
          = \lambda (u,v)_{L^2(\eta)}
        \right\rangle \\
        &\,\,\, \oplus \,
        \bigoplus_{j=2}^{2n} \left\langle \left. u \in H^1_{\Hhat}(\Omega_j,\eta) \; \right| \;
          \exists \lambda>0 \;\; \forall v \in H^1_{\Hhat}(\Omega_j,\eta) \;\;
          B_{\Lhat_\eta,\Omega_j}[u,v]
          = \lambda (u,v)_{L^2(\eta)}
        \right\rangle,
    \end{aligned}
  \end{equation}
where in the first equation we take the direct sum within $H^1_{\Hhat}(\Sk,\eta)$,
while in the second we use the abstract direct sum,
and in both equations angled brackets indicate the linear span in $H^1_{\Hhat}(\Sk,\eta)$.

Using the variational characterization of
eigenvalues and the unique-continuation principle 
we find as in \cite{MR} that
(i) orthogonal projection in $H^1_{\Hhat}(\Sk,\eta)$ onto the subspace
spanned by the $-\Lhat_\eta$ eigenfunctions with strictly negative eigenvalues
has injective restriction to $V$ and
(ii) orthogonal projection
in $\bigoplus_{j=1}^{2N} H^1_{\Hhat}(\Omega_j,\eta)$ onto $W$
has injective restriction (after precomposition with the obvious inclusion)
to the subspace of $H^1_{\Hhat}(\Sk,\eta)$
spanned by $-\Lhat_\eta$ eigenfunctions with nonpositive eigenvalues.

According to Lemma \ref{hemisphere} below
$-\Lhat_\eta$ on $H^1_{0, \Hhat}(\Omega_j,\eta)$
(with Dirichlet condition)
has one-dimensional kernel and no strictly negative eigenvalues,
so from (i) in the preceding paragraph we deduce that
$-\Lhat_\eta$ on $H^1_{\Hhat}(\Sk,\eta)$
has at least $2n-1$ strictly negative eigenvalues,
counted with multiplicity.
On the other hand Lemma \ref{hemisphere}
also states that $-\Lhat_\eta$ on $H^1_{\Hhat}(\Omega_j,\eta)$
with Neumann condition
has trivial kernel and precisely one simple negative eigenvalue,
so from (ii) in the preceding paragraph we deduce that
$-\Lhat_\eta$ on $H^1_{\Hhat}(\Sk,\eta)$
has no more than $2n$ nonnegative eigenvalues, counted with multiplicity.
Thus $\Lhat_\eta$ has nullity at most one on $H^1_{\Hhat}(\Sk,\eta)$.
Since the vertical component of the Gauss map
is a Jacobi field for $\Sk$ and
is $\Hhat$ but not $\Grphat$ equivariant,
we see that the $H^1_{\Grphat}(\Sk,\eta)$ kernel of $\Lhat_\eta$ is indeed trivial.
\end{proof}

\begin{lemma}
\label{hemisphere}
For each $j \in \Z$, with notation as in \ref{Omegadef} and the proof of Proposition \ref{nokernel},
  \begin{enumerate}[(i)]
    \item $-\Lhat_\eta$ on $H^1_{0,\Hhat}(\Omega_j,\eta)$
      (with Dirichlet boundary condition)
      has $1$-dimensional kernel and no strictly negative eignevalues, and

    \item $-\Lhat_\eta$ on $H^1_{\Hhat}(\Omega_j,\eta)$
      with Neumann boundary condition
      has trivial kernel and exactly one simple negative eigenvalue.
  \end{enumerate}
\end{lemma}

\begin{proof}
The hemisphere $(\Omega_j,\eta)$ with its pole deleted is conformal,
via stereographic projection (from the antipodal pole)
and a logarithm, to the standard half-cylinder with flat metric.
Concretely, using polar coordinates on the unit disc pulled back
by the same stereographic projection to $\Omega_j$, we have
  \begin{equation}
    \eta = \frac{4r^2}{(1+r^2)^2}\left(r^{-2}dr^2 + d\theta^2\right).
  \end{equation}
We find (recall \ref{potential}) that for the corresponding operator
  \begin{equation}
    \Lhat_{cyl} := \frac{4r^2}{(1+r^2)^2}\Lhat_\eta
    = (r\partial_r)^2 + 8(k-1)^2 \frac{r^{2k-2}}{(r^{2k-2}+1)^2} + \partial_\theta^2
  \end{equation}
we have
  \begin{equation}
    \begin{aligned}
      &\Lhat_{cyl}
      =
      A_-A_+ + (k-1)^2 + \partial_\theta^2 \mbox{ and} \\
      &\Delta_{cyl}:=(r\partial_r)^2 + \partial_\theta^2
      =
      A_+A_- + (k-1)^2 + \partial_\theta^2, \mbox{ where} \\
      &A_\pm
      =
      r\partial_r \pm (k-1)\frac{r^{2k-2}-1}{r^{2k-2}+1}.    
    \end{aligned}
  \end{equation}

Evidently $\Lhat_{cyl}A_-=A_-\Delta_{cyl}$,
so whenever $w$ is an eigenfunction of the flat Laplacian
$\Delta_{cyl}$,
$A_-w$ is an eigenfunction (when nonzero) of $\Lhat_{cyl}$ with the same eigenvalue.
Moreover $A_-r^\lambda$ and $A_-r^{-\lambda}$ are linearly independent
for every real $\lambda \geq 0$ except $\lambda=k-1$ and obviously $\lambda=0$.
Corresponding to the latter exception we have for $A_-A_+ + (k-1)^2$ the linearly independent eigenfunctions
$A_-1$ and $A_-\ln r$ with eigenvalue $0$.
As for the former exception, 
noting that we have already accounted for the eigenfunction
$A_-r^{k-1}=-A_-r^{1-k}$ of $A_-A_+ + (k-1)^2$ with eigenvalue $(k-1)^2$
and that $\Delta_{cyl}A_+=A_+\Lhat_{cyl}$,
we find that nonzero $w$ solving $A_+w=r^{k-1}+r^{1-k}$ is another, independent such eigenfunction.

Thus we deduce that an eigenfunction for $-\Lhat_{cyl}$
of the form $v_\lambda(r)e^{i\ell\theta}$
with eigenvalue $\ell^2-\lambda^2$
has radial factor $v_\lambda(r)$ a linear combination of $u_{\abs{\lambda}}(r)$
and $u_{-\abs{\lambda}}(r)$ given by 
  \begin{equation}
      u_\lambda(r) 
      :=
      \left(r\partial_r - (k-1)\frac{r^{2k-2}-1}{r^{2k-2}+1}\right)
        r^\lambda
      = \left(\lambda - (k-1)\frac{r^{2k-2}-1}{r^{2k-2}+1}\right)r^\lambda
  \end{equation}
or a linear combination of $u_0(r)$ and
  \begin{equation}
    u_{0'}(r) := 1 - (k-1)\frac{r^{2k-2}-1}{r^{2k-2}+1} \ln r
  \end{equation}
in case $\lambda = 0$
or a linear combination of $u_{k-1}(r)=-u_{1-k}(r)$ and
  \begin{equation}
    u_{(k-1)'}(r) := r^{k-1}\frac{r^{2k-2}-r^{2-2k}+4(k-1)\ln r}{r^{2k-2}+1}
  \end{equation}
in case $\abs{\lambda} = k-1$.

Separating variables, we need only consider eigenfunctions of the above form,
and, because of the rotational symmetries imposed,
$\ell$ must take values in $k\Z$.
Now suppose $\Lhat_\eta v_\lambda(r)e^{i\ell\theta} = 0$ on $\Omega_j$.
Then $\Lhat_{cyl} v_\lambda(r)e^{i\ell \theta}=0$ as well,
on $\Omega_j$ punctured at its pole $\{0\}$,
so $\lambda = \abs{\ell}$ and $v_\lambda(r)=c_1u_\lambda(r)+c_2u_{-\lambda}(r)$
for some constants $c_1$ and $c_2$,
unless $\ell=0$, in which case $v_\lambda(r)=c_1u_0(r)+c_2u_{0'}(r)$.
(The other exceptional case of $\lambda=k-1$ is excluded by the symmetries.)

If we impose Dirichlet conditions,
we find $c_1=c_2$ when $\lambda>0$,
but then $v_\lambda$ is singular at $r=0$ unless $c_1=c_2=0$,
which must therefore hold, since $v_\lambda(r)e^{i\ell\theta}$
is an eigenfunction on all $\Omega_j$.
On the other hand $u_0(1)=0$ while $u_{0'}(1)=1$,
so $c_2=0$ when $\lambda=0$.
Thus Dirichlet $\Lhat_\eta$ has kernel spanned by
$u_0(r)=\frac{r^{2k-2}-1}{r^{2k-2}+1}$,
confirming the nullity asserted in (i).
If we instead impose Neumann conditions,
since $\partial_ru_{\lambda}(1)=\lambda^2-(k-1)^2$
(again $\lambda=k-1$ is excluded by the symmetries),
we find $c_1=-c_2$ when $\lambda>0$,
so again we need $c_1=c_2=0$ to avoid a singularity at $r=0$.
On the other hand $\partial_ru_{0'}(1)=0$
while $\partial_ru_{0}(1)=-(k-1)^2$,
so in this case $c_1=0$, but $u_{0'}$ is also singular at $r=0$.
Thus Neumann $\Lhat_\eta$ indeed has trivial kernel,
as claimed in (ii).

For $\lambda \neq 0$ one cannot expect agreement of 
solutions to $\Lhat_\eta u=\lambda u$
and to $\Lhat_{cyl} u=\lambda u$,
but the variational characterization of eigenvalues
reveals that the number (counting multiplicity)
of strictly negative (Dirichlet or Neumann) eigenvalues
will agree for the two operators,
at least on compact subsets of $\Omega_j \backslash \{0\}$,
where the conformal factor is bounded with bounded inverse.
Furthermore, again using the variational characterization of eigenvalues
one sees that for $\epsilon>0$ sufficiently small,
the number (counted with multiplicity)
of strictly negative eigenvalues for $\Lhat_\eta$
on $\Omega_j$ with Dirichlet condition on the equator
is the same as the number of strictly negative eigenvalues of $\Lhat_\eta$
on $\Omega_j$ less a spherical cap $B_\epsilon(0)$ with radius $\epsilon$ and
center the pole $0$, imposing Dirichlet conditions on both the equator
and boundary of the cap.
(To show the former number is at least the latter
extend test functions vanishing on $\partial B_\epsilon(0)$
to test functions vanishing on $B_\epsilon(0)$;
for the reverse inequality use a logarithmic cut-off,
identically $0$ on $B_\epsilon(0)$
and identically $1$ on $B_{\sqrt{\epsilon}}(0)$.)
Likewise, assuming $\epsilon$ small enough,
the number of strictly negative eigenvalues of $\Lhat_\eta$
on $\Omega_j$ with Neumann condition on the equator
is the same as the number of strictly negative eigenvalues
of $\Lhat_\eta$ on $\Omega_j \backslash B_\epsilon(0)$
with Neumann condition on the equator and Dirichlet condition on $\partial B_\epsilon(0)$.

Now suppose $v_\lambda(r)e^{i\ell \theta}$ is an eigenfunction of $\Lhat_\eta$
on $\Omega_j \backslash B_\epsilon(0)$
with strictly negative eigenvalue,
imposing either of the above boundary conditions.
Then $\abs{\lambda}>\abs{\ell}$,
so in particular $\lambda \neq 0$.
If $\abs{\lambda}=k-1$,
then $v_\lambda=c_1u_{k-1}+c_2u_{(k-1)'}$.
If we impose the Dirichlet condition on the equatorial circle $r=1$,
then, since $u_{k-1}(1)=k-1$ 
and $u_{(k-1)'}(1)=0$,
we must have $c_1=0$,
but $\lim_{r \to 0} u_{(k-1)'}(r)=-\infty$,
so the Dirichlet condition on $r=\epsilon$ means $c_2=0$ as well.
Thus there are no such Dirichlet eigenfunctions.
If instead we impose the Neumann condition on $r=1$,
then, since $\partial_r u_{k-1}(1)=0$ 
and $\partial_r u_{(k-1)'}(1)=4(k-1)$,
we need $c_2=0$,
but $u_{k-1}(r)$ vanishes only at $r=0$,
so the Dirichlet condition on $r=\epsilon$ means $c_1=0$ too.
Thus there are no such Neumann eigenfunctions either.

Now assume $\abs{\lambda} \neq k-1$.
Then $v_\lambda=c_1u_{\abs{\lambda}}+c_2u_{-\abs{\lambda}}$.
As when studying the kernel above, imposition of the Dirichlet condition $r=1$
forces $c_1=c_2$,
but a quick calculation shows that imposition of the Dirichlet condition on $r=\epsilon$
then requires $\lambda \tanh [\lambda \ln \epsilon] = (k-1) \tanh [(k-1) \ln \epsilon]$.
Since for every real $c \neq 0$ the function $x \tanh cx$ is even
and on $[0,\infty)$ strictly monotonic,
this last condition implies $\lambda=k-1$,
contradicting the initial assumption of the paragraph
and completing the proof of (i).
Imposing instead the Neumann condition on $r=1$ forces $c_1=-c_2$,
so the Dirichlet condition on $r=\epsilon$ now demands
$\lambda \coth [\lambda \ln \epsilon] = (k-1) \tanh [(k-1) \ln \epsilon]$.
For any $\epsilon>0$ the function on the left is even in $\lambda$,
strictly monotonic in $\lambda$ on $[0,\infty)$,
and has limit $(\ln \epsilon)^{-1}$ as $\lambda$ goes to $0$; 
moreover, as $\epsilon$ tends to $0$, the right-hand side goes to $1-k$.
Thus this equation has exactly one solution, completing the proof of (ii).
\end{proof}

\begin{remark}
\label{r:nk} 
A simpler proof of 
\ref{hemisphere} is also possible using the $h$ metric instead of the $\eta$ metric and without reference to the $A_\pm$ 
operators. 
\ref{towsol} can also be proved without using the $\eta$ metric. 
\end{remark} 

\subsection*{Approximate solutions on tori}

Now we state some estimates
for solutions to the Poisson equation $\Delta_{\geuc}u=f$
on the Euclidean strip
  \begin{equation}
  \label{TXdef}
    T_X:=(0,X\pi) \times \R
  \end{equation}
of given width $X>0$,
with $u$ subject to Dirichlet data
and $f$ odd under reflection through
the horizontal line $y=jY\pi$ for given $Y>0$ and every $j \in \Z$.
We set
  \begin{equation}
  \label{Kgrphatdef}
    \Kgrphat_Y:=\left\langle \refle_{y=0}^\pi, \, \refle_{y=Y\pi}^\pi \right\rangle.
  \end{equation}

In the applications to follow, $X$ will tend to infinity with $m$, while $Y$
will be bounded independently of $m$, so it is important that the estimates
here do not depend on $X$.
The additional decay estimate included in the proposition will be necessary
to guarantee convergence of the iterative scheme used to construct global solutions
on the initial surfaces.

\begin{prop}
\label{torsol}
With notation as in the preceding paragraph,
given $X>Y>0$ and $\beta \in (0,1)$, there exists a linear map
  \begin{equation}
    \Rtor_{X,Y}: 
    C^{0,\beta}_{\Kgrphat_Y}(T_X ,\geuc)
    \to 
    C^{2,\beta}_{\Kgrphat_Y}(T_X,\geuc)
  \end{equation}
and there exists a constant $C>0$---depending on $\beta$ and $Y$
but not on $X$---such that if $f \in C^{0,\beta}(T_X,\geuc)$ and $u=\Rtor_{X,Y}f$,
then $\Delta_{\geuc}u = f$, $u$ vanishes on $\partial T_X$, and
  \begin{equation}
  \label{torschauder}
    \norm{u : C^{2,\beta}(T_X,g_E)} 
    \leq
    C\norm{f: C^{0,\beta}(T_X,g_E)}.
  \end{equation}
Moreover, if $f$ vanishes outside $[A,B] \times \R$ for $1 < A < B < X\pi-1$, then
  \begin{equation}
    \begin{aligned}
      &\norm{u:
        C^{2,\beta}\left((0,A-1) \times \R, \, \geuc, \, e^{(x-A)/Y}\right)}
      \leq C \norm{f: C^{0,\beta}(T_X,\geuc)} \mbox{ and } \\
      &\norm{u:
        C^{2,\beta}\left((B+1,X\pi) \times \R, \, \geuc, \, e^{(B-x)/Y}\right)}
      \leq C \norm{f: C^{0,\beta}(T_X,\geuc)},
    \end{aligned}
  \end{equation}
where $x$ is the coordinate on the $[0,X\pi]$ factor of $T_X$.
\end{prop}

\begin{proof}
Define $\Rtor_{X,Y} f$ to be the Dirichlet solution $u$ to the Poisson equation
$\Delta_{g_E}u = f$.
Then $u \in C^{2,\beta}_{\Kgrphat_Y}(T_X,\geuc)$ and
for each $p \in T_X$
  \begin{equation}
    \norm{u}_{C^{2,\beta}(B_1(p),\geuc)} \leq C \left(
      \norm{u}_{C^0(B_2(p))} + \norm{f}_{C^{0,\beta}(B_2(p),\geuc)}\right),
  \end{equation}
where $B_r(p)$ is the intersection with $T_X$ of the Euclidean disc with
center $p$ and radius $r$. Defining
  \begin{equation}
    \begin{aligned}
      &f_n(x) := \sqrt{\frac{2}{Y\pi}} \int_0^{Y\pi} f(x,y) \sin \frac{n y}{Y} \, dy \mbox{ and} \\
      &u_n(x) := \sqrt{\frac{2}{Y\pi}} \int_0^{Y\pi} u(x,y) \sin \frac{n y}{Y} \, dy,
    \end{aligned}
  \end{equation}
for each positive integer $n$
we have
  \begin{equation}
    \ddot{u}_n(x) - \frac{n^2}{Y^2}u_n(x) = f_n(x),
  \end{equation}
for which equation one finds Dirichlet Green's function
  \begin{equation}
    G_n(x,x') = \frac{-Y}{n \sinh \frac{nX\pi}{Y}}
    \begin{cases}
      \sinh \frac{n(X\pi-x')}{Y} \sinh \frac{nx}{Y} \mbox{ if } x \leq x' \\
      \sinh \frac{nx'}{Y} \sinh \frac{n(X\pi-x)}{Y} \mbox{ if } x \geq x'.
    \end{cases}
  \end{equation}

Since
  \begin{equation}
    \begin{aligned}
      \frac{n^2}{Y^2} \sinh \frac{nX\pi}{Y} \int_0^{X\pi} \abs{G(x,x')} \, dx'
      &\leq 
        \sinh \frac{n(X\pi-x)}{Y} \cosh \frac{nx}{Y} - \sinh \frac{n(X\pi-x)}{Y} \\
        &\;\;\;\; + \sinh \frac{nx}{Y} \cosh \frac{n(X\pi-x)}{Y} - \sinh \frac{nx}{Y} \\
      &\leq
        \sinh \frac{nX\pi}{Y},
    \end{aligned}
  \end{equation}
we have
  \begin{equation}
  \label{modest}
    \norm{u_n}_{C^0([0,X\pi])} \leq \frac{Y^2}{n^2}\norm{f_n}_{C^0([0,X\pi])}
    \leq \frac{\sqrt{2Y^5\pi}}{n^2}\norm{f}_{C^0(T_X}),
  \end{equation}
and so
  \begin{equation}
    \norm{u}_{C^0(T_X)}
    \leq 
    \sum_{n=1}^\infty \sqrt{\frac{2}{Y\pi}}\abs{u_n}
    \lesssim
    Y^2\norm{f}_{C^0(T_X)},
  \end{equation}
which upgrades the local Schauder estimates above to the first inequality
asserted in the proposition.
If moreover $f$ vanishes outside $[A,B] \times \R$, then
  \begin{equation}
    \begin{aligned}
      &u_n|_{[0,A]}(x) = \frac{u_n(A) \sinh \frac{nx}{Y}}{\sinh \frac{nA}{Y}} \mbox{ and} \\
      &u_n|_{[B,X\pi]}(x) = 
        \frac{u_n(B) \sinh \frac{n(X\pi-x)}{Y}}{\sinh \frac{n(X\pi-B)}{Y}},
    \end{aligned}
  \end{equation}
establishing in conjunction with \ref{modest} the decay estimates.

\end{proof}

\subsection*{Global solutions}

The final task of this section is to apply
Propositions \ref{towsol} and \ref{torsol}
iteratively
on the extended standard regions 
to prove existence and obtain estimates
of global solutions to the equation $\Lcal u=f$ on each initial surface.

\begin{prop}
\label{Rcal}
Fix $\beta \in (0,1)$ and
data $(k,n_1,n_2,\sigma)$
or $(k,n,n'_1,n'_{-1},\sigma'_1,\sigma'_{-1})$ for an initial surface.
There is a positive integer $m_0$ such that for every $m \geq m_0$
and for every initial surface $\Sigma$ defined by
the corresponding data
there exists a linear map
  \begin{equation}
    \Rcal: C^{0,\beta}_{\Grp}(\Sigma,m^2g) \to C^{2,\beta}_{\Grp}(\Sigma,m^2g)
  \end{equation}
and there exists a constant $C>0$---independent of $m$---such that if
$f \in C^{0,\beta}_{\Grp}(\Sigma,m^2g)$, then $\Lcal \Rcal f = m^2f$ and
  \begin{equation}
    \norm{\Rcal f : C^{2,\beta}_{\Grp}(\Sigma, m^2g)}
    \leq
    C \norm{f : C^{0,\beta}_{\Grp}(\Sigma, m^2g)}.
  \end{equation}
\end{prop}

\begin{proof}
For $T \in \comptor(\Sigma)$ with boundary circles $C, D \in \intcirc(\Sigma)$,
set
  \begin{equation}
    \begin{aligned}
      &A_T
      :=
      \frac{m}{\pi}\left(d_{\gsph}(C,D)-\frac{b}{\mC}-\frac{b}{m_{_D}}\right) \mbox{ and} \\
      &B_T
      :=
      \begin{cases}
        1/k \mbox{ for $\Sigma$ of type $M$} \\
        1/2k \mbox{ for $\Sigma$ of type $N$}
      \end{cases}
    \end{aligned}
  \end{equation}
and define the
diffeomorphism (recalling definitions \ref{STdef} and \ref{TXdef})
  \begin{equation}
    X_{_T}: T_X / \left\langle \transe_{y\text{-axis}}^{2m\pi} \right\rangle \to S[T],
  \end{equation}

the cutoff function $\psi_{_T} \in C^\infty(S[T])$,
and the linear map
  \begin{equation}
    \Rcal_T: C_{\Grp}^{0,\beta}(S[T],m^2g) \to C_{\Grp}^{2,\beta}(S[T],m^2g)
  \end{equation}
by
  \begin{equation}
    \begin{aligned}
      &\Rcal_Tf := {X_{_T}^*}^{-1}\Rtor_{A_T,B_T} X_{_T}^* f, \\
      & \psi_{_T} :=
        \frac{1}{2}{X_{_T}^*}^{-1}\left(\cutoff{0}{1} \circ x
        + \cutoff{A_T}{A_T-1} \circ x \right), \mbox{ and } \\
      &X_{_T} := \varpi_T^{-1} \circ \kappa,
    \end{aligned}  
  \end{equation}
where
$\kappa: \left(T_X / \left\langle \transe_{y\text{-axis}}^{2m\pi} \right\rangle,\geuc\right)
\to (T_b, m^2\gsph)$
is any isometry
mapping $y=0$ to a scaffold circle on the torus containing $T_b$.

For $C \in \intcirc(\Sigma)$ recall the diffeomorphism
  \begin{equation}
    \xC: \SkC(a) \to S[C]
  \end{equation}
and define
the cutoff function $\psi_{_C} \in C^\infty(S[C])$
and the linear map
  \begin{equation}
    \Rcal_C: C_{\Grp,c}^{0,\beta}(\phiC(\SkC(b+1)),m^2g)
    \to C_{\Grp}^{2,\beta}(S[C],m^2g)
  \end{equation}
by
  \begin{equation}
    \begin{aligned}
      &\Rcal_C := \frac{m^2}{\mC^2} {\xC^*}^{-1}
        \RtowkmC \xC^* f \mbox{ and} \\
      &\psi_{_C}:={\xC^*}^{-1}
        \left(\cutoff{a}{a-1} \circ \rr \right).
    \end{aligned}
  \end{equation}

Next, given $f \in C^{0,\beta}_{\Grp}(\Sigma, m^2g)$, let
  \begin{equation}
    \begin{aligned}
      &f_1 := \left(1-\sum_{T \in \comptor(\Sigma)} \psi_{_T}\right)f
        + \sum_{T \in \comptor(\Sigma)} 
        [\psi_{_T},m^{-2}\Lcal]\Rcal_Tf|_{S[T]}, \\
      &f_2 := \sum_{C \in \intcirc(\Sigma)} [\psi_{_C},m^{-2}\Lcal]
        \Rcal_Cf_1|_{S[C]}, \mbox{ and} \\
      &\widetilde{\Rcal} f := \sum_{T \in \comptor(\Sigma)} \psi_T\Rcal_T (f+f_2)|_{S[T]}
        + \sum_{C \in \intcirc(\Sigma)} \psi_C\Rcal_C f_1|_{S[C]}.
    \end{aligned}
  \end{equation}

The idea behind the definition of $\widetilde{\Rcal}f$ is as follows.
First we construct approximate solutions on each toral region
(the $\Rcal_T f|_{S[T]}$ terms) and cut them off smoothly.
These solutions are only approximate since we have obtained them by applying
the solution operator for the model problem on the Euclidean strip,
and the resulting error is controlled by the deviation of the initial surface's
geometry from the model geometry.
Additional error, supported in the tower regions,
is created by cutting off the approximate solution with $\psi_{_T}$.
We know only that its size is controlled by that of the original $f$,
and we account for it in $f_1$
along with the restriction of the original $f$ to the tower regions,
where we next construct and cut off approximate solutions in a similar fashion.
Again there is error controlled by the geometry and also cutoff error,
for which we have no better bound than the norm of $f$
but which is supported inside the toral regions
far from their boundary,
so we can construct an approximate solution to correct for them
and apply the decay estimate in \ref{towsol}.

More precisely we now check that
  \begin{equation}
    \begin{aligned}
      f - m^{-2}\Lcal \widetilde{\Rcal}f =
      &\sum_{C \in \intcirc(\Sigma)}
        \left(
          \frac{\mC^2}{m^2}
          {\xC^*}^{-1}\left(\Delta_{\ghat_{_C}}+\abs{\Ahat_{_C}}^2\right)\xC^*
          -m^{-2}\Lcal
        \right)
        \Rcal_Cf_1|_{S[C]} \\
      &+\sum_{T \in \comptor(\Sigma)}
        \left( {X_{_T}^*}^{-1}\Delta_{\geuc}X_T^*-m^{-2}\Lcal \right)
        \Rcal_T
        \left(f+f_2\right)|_{S[T]} \\
      &+\sum_{T \in \comptor(\Sigma)}
        \left[ \psi_{_T}, m^{-2}\Lcal \right] \Rcal_T
        f_2|_{S[T]},
    \end{aligned}
  \end{equation}
so, noting that $f_2|_{S[T]}$ is supported far away from $\partial S[T]$,
we find from \ref{stdtor}, \ref{stdtow}, \ref{towsol}, and \ref{torsol}
  \begin{equation}
    \norm{f - m^{-2} \Lcal \widetilde{\Rcal} f}_{C^{0,\beta}(\Sigma,m^2g)}
    \leq 
    \left(
      C(b)m^{-1/2} + Ce^{-b} + C(b)e^{-a} 
    \right)
    \norm{f}_{C^{0,\beta}(\Sigma,m^2g)},
  \end{equation}
where $C(b)$ is a constant depending on $b$ but not on $m$
and where $C$ is a constant depending on neither $b$ nor $m$.
Thus we may at this stage fix $b$ (finally determining the extent of the
toral regions) sufficiently large in terms of $C$ and then take $m$ sufficiently large
in terms of $C(b)$ so as to ensure that $m^{-2}\Lcal \widetilde{\Rcal}$ is invertible.
The proof is then concluded by taking
$\Rcal = \widetilde{\Rcal}\left(m^{-2}\Lcal \widetilde{\Rcal}\right)^{-1}$.
\end{proof}

\section{The main theorem}
\label{nonlinear}

Recall that given an initial surface $\Sigma$ with defining embedding 
$X: \Sigma \to \Sph^3$ and a function $u: \Sigma \to \R$,
we have defined the map $X_u: \Sigma \to \Sph^3$
by $X_u(p) = \exp_{X(p)} u(p)\nu(p)$, $\exp$ being the exponential map
on $\Sph^3$ and $\nu$ a global unit normal for the initial surface.
For $u \in C^2_{loc}(\Sigma)$ sufficiently small $X_u$ is an immersion
with well-defined mean curvature $\Hcal[u]$ relative to the global unit normal
$\nu_u$ having positive inner product with the parallel translates
of $\nu$ along the geodesics
it generates. We now prove the main theorem by solving $\Hcal[u]=0$.

\begin{theorem}
\label{mainthm}
Given data (a) $(k,n_1,n_2,\sigma)$
or (b) $(k,n,n'_1,n'_{-1},\sigma'_1,\sigma'_{-1})$
for an initial surface (recalling \ref{initsdef}),
there exist $m_0>0$ and $C>0$ such that 
whenever $m>m_0$,
the initial embedding
$X: \Sigma \to \Sph^3$
corresponding to the data 
can be perturbed to a minimal embedding
$X_u: \Sigma \to \Sph^3$
by a function 
$u \in C_{\Grp}^\infty(\Sigma)$
(depending on $m$)
that satisfies the estimate $\norm{u: C^2(\Sigma,m^2g)} \leq Cm^{-3/2}$.
Here $\Grp$ is either (a) $\Grp_{k,m}$ or (b) $\Grp'_{k,m}$ (recalling \ref{symm}).
In particular $X_u(\Sigma)$ has the same genus as $\Sigma$ (see \ref{initsprop}),
is invariant under $\Grp$,
and contains the scaffolding (a) $\Ckm$ or (b) $\Ckmp$ (recalling \ref{scaff}).
Moreover, in the complement in $\Sph^3$ of any tubular neighborhood of the circles of intersection
of the initial configuration (a) $\Wcal_k$ or (b) $\Wcal'_k$, for $m$ sufficiently large $X_u(\Sigma)$
is the graph over some subset of the initial configuration
of a smooth function converging smoothly to $0$ as $m \to \infty$.
\end{theorem}

\begin{proof}
Fix $\beta \in (0,1/2)$.
By \ref{stdtor} and \ref{stdtow} the initial mean curvature satisfies
\begin{equation}
  m^{-2}\norm{\Hcal[0] : C^{2,2\beta}(\Sigma,m^2g)} \leq Cm^{-3/2}
\end{equation}
for a constant $C$ independent of $m$. Setting
  \begin{equation}
    u_0 = -\Rcal m^{-2}\Hcal[0],
  \end{equation}
then \ref{Rcal} implies that
  \begin{equation}
    \norm{u_0 : C^{2,2\beta}(\Sigma,m^2g)} \leq Cm^{-3/2}
  \end{equation}
for a (possibly different) constant $C$ independent of $m$.
The function $u_0$ represents the first-order correction to the initial surface.
To complete the perturbation we need to estimate the nonlinear part of $\Hcal$ near $0$,
defined by
  \begin{equation}
    \Qcal[u] := \Hcal[u] - \Hcal[0] - \Lcal u.
  \end{equation}

To proceed efficiently we consider the blown-up metric $m^2\gsph$ on $\Sph^3$.
Given $u: \Sigma \to \R$ we can define
$X_{u,m^2\gsph}: \Sigma \to \Sph^3$ by
$X_{u,m^2\gsph}(p) := \exp_{X(p)}^{m^2\gsph} u(p)\nu_{m^2\gsph}(p)$,
where $\exp^{m^2\gsph}$ is the exponential map on $(\Sph^3,m^2\gsph)$
and $\nu_{m^2\gsph}$ is the $m^2\gsph$ unit normal for $\Sigma$
parallel to $\nu$;
of course $\exp^{m^2\gsph}=\exp$, $\nu_{m^2\gsph} = m^{-1}\nu$, and
$X_u = X_{mu,m^2\gsph}$. For $u \in C^2_{loc}$ sufficiently small
we can define also $\Hcal_{m^2\gsph}[u]$ to be the mean curvature
of $X_{u,m^2\gsph}$ relative to $m^2\gsph$ (and $m^{-1}\nu_u$). Obviously
  \begin{equation}
    \Hcal[u] = m\Hcal_{m^2\gsph}[mu],
  \end{equation}
so
  \begin{equation}
    m^{-2}\Qcal[u]
    =
    m^{-2}\int_0^1 \int_0^t \frac{d^2}{ds^2}\Hcal[su] \, ds \, dt
    = m^{-1} \int_0^1 \int_0^t \frac{d^2}{ds^2}\Hcal_{m^2g_S}[smu] \, ds \, dt.
  \end{equation}
Now, if $mu \in C^2_{loc}(\Sigma)$ is sufficiently small in terms of the Riemannian
curvature of $(\Sph^3,m^2\gsph)$ and the second fundamental form of $X$
relative to $m^2\gsph$, then $X_u$ will be an immersion, 
$\Hcal[u]$ will be well-defined, and moreover
  \begin{equation}
    \sup_{s \in [0,1]} \norm{\frac{d^2}{ds^2}\Hcal_{m^2g_S}[smu] :
      C^{0,2\beta}(\Sigma,m^2g)}
    \leq C\norm{mu: C^{2,2\beta}(\Sigma,m^2g)}^2,
  \end{equation}
where $C$ is a constant controlled by finitely many covariant derivatives 
of the Riemannian curvature of the ambient space $(\Sph^3,m^2\gsph)$ 
and finitely many covariant derivatives
of the second fundamental form of $X$ relative to $m^2\gsph$. Of course
the Riemannian curvature of $(\Sph^3,m^2\gsph)$ is bounded uniformly in $m$
(tending to $0$ in fact) and all of its derivatives vanish; while $X$ itself
depends on $m$, each derivative of its second fundamental form, relative to $m^2\gsph$,
is bounded independently of $m$.

Consequently, if $B$ is the closed ball of radius $m^{-7/4}$ in 
$C^{2,2\beta}(\Sigma,m^2g)$ and $v \in B$, we have
  \begin{equation}
    \norm{m^{-2} \Qcal[u_0+v]: C^{2,2\beta}(\Sigma,m^2g)} \leq Cm^{-2}.
  \end{equation}
Evidently then, taking $m$ large enough, $F(v)=-m^{-2}\Rcal \Qcal[u_0+v]$ defines a map
$F: B \to B$ which is continuous with respect to the $C^{2,\beta}(\Sigma,m^2g)$
norm on $B$ as well as the $C^{2,2\beta}(\Sigma,m^2g)$ norm, so by the
Schauder fixed point theorem admits a fixed point $v_0 \in B$.
Accordingly $\Lcal v_0 = -\Qcal[u_0+v_0]$ and
  \begin{equation}
    \Hcal[u_0+v_0] = \Hcal[0] + \Lcal u_0 + \Lcal v_0 + \Qcal[u_0+v_0] = 0.
  \end{equation}
The higher regularity of $u=u_0+v_0$ then follows immediately,
and the $C^0$ decay estimate of $mu$ ensures embeddedness.
\end{proof}

\section{Further results and discussion} 
\label{further}

\subsection*{Highly symmetric constructions with obstructions}

In this subsection we briefly outline a highly symmetric construction where the symmetry imposed is not so great that there are no obstructions. 
The obstruction space is nontrivial but of finite dimension independent of the symmetries and the genus of the surfaces constructed. 
The construction can easily be explained in terms of the earlier presentation:
the initial configuration used is $\Wcal'_k$  and the symmetry group imposed is $\Grp_{k,m}$ (and not $\Grp'_{k,m}$). 
This corresponds to using the scaffolding $\Ckm\subset\Wcal_k$. 
The towers desingularizing $C_1$ and $C_2$ are then symmetric enough
that they carry no kernel. 
The construction in this respect can proceed as the earlier one in \ref{mainthm}. 
On the other hand the towers desingularizing the circles $C'_j\subset\T'\cap\T_j$ 
are classical Scherk singly periodic surfaces and the symmetries imposed fix $\T_j$ but not $\T'$. 
This situation is similar to many recent constructions \cite{Nguyen,KKM,KapLi}  
where there is enough symmetry to simplify the obstruction space in comparison to 
the more general situation in \cite{KapJDG}, 
but not enough to render it trivial as in \ref{mainthm}. 

More precisely we have a
two-dimensional kernel, one dimension for each circle of 
intersection $C'_1$ and $C'_2$. 
(Note that modulo the symmetries these are the only circles of intersection besides $C_1$ and $C_2$). 
There are no circles in the scaffolding contained in $\T'$ and therefore $\T'$ is not held fixed 
by the construction. 
We introduce then two continuous parameters in the construction, $x_1$ and $x_2$. 
$C'_1$ is replaced by a parallel copy on $\T_1$ at (signed) distance $x_1$ and 
similarly $C'_2$ 
is replaced by a parallel copy on $\T_2$ at (signed) distance $x_2$.  
By the symmetries then all $C'_j$ are appropriately replaced also. 
$\T'$ is a union of annuli with boundaries the $C'_j$.
These are replaced then by minimal graphs so that the new annuli span the $C'_j$'s. 
This way $\T'$ is replaced by a new torus with derivative discontinuities along its circles 
of intersection with the $\T_j$'s. 
The construction of the initial surfaces then proceeds as usual by using towers appropriately. 

Note that modulo the symmetries there are four circles which get desingularized: 
$C_1,C_2$ and the (perturbed to new positions) $C'_1,C'_2$. 
Following the same conventions as in Section \ref{Sinit} we denote by $n_1,n_2,n'_1,n'_2$ the 
number of half periods the desingularizing towers will have 
between successive circles of reflection in $\Ckm$ along $C_1,C_2,C'_1,C'_2$ respectively.  
This together with three alignment parameters $(\sigma,\sigma'_1,\sigma'_2)$ 
and the continuous parameters $x_1$ and $x_2$ 
determine the initial surfaces. 
We have the following.

\begin{theorem}
\label{thm-2}
Given data $(k,n_1,n_2,n'_1,n'_2, \sigma,\sigma'_1,\sigma'_2)$ 
for an initial surface as outlined above 
there exists $m_0>0$ such that 
whenever $m>m_0$,
one of the initial surfaces (for some appropriate values of $x_1,x_2$) 
described above
can be perturbed to a minimal surface which contains $\Ckm$, 
is symmetric under the action of $\Grp_{k,m}$, 
and has genus 
$k(k-1)m(n_1+n_2) + k^2m(n'_1+n'_{2}) +  1$. 
Moreover as $m\to\infty$ the minimal surfaces converge as varifolds to $\Wcal'_k$.   
\end{theorem}

\begin{proof}
The proof combines the arguments for \ref{mainthm} 
with the arguments for constructions like in \cite{Nguyen,KKM,KapLi}.   
Details will be presented elsewhere. 
\end{proof}

\subsection*{Corollaries of a general desingularization theorem} 

In this subsection we discuss corollaries in our setting 
of a general desingularization theorem announced in \cite[Theorem F]{KapClay} and \cite[Theorem 3.1]{KapSchoen}. 
The statement of this theorem is motivated in \cite[section 14]{KapClay}, 
and its proof is outlined in detail in \cite[sections 5-8]{KapSchoen} 
and will be presented in detail in \cite{Kap:compact}. 
We will refer to this theorem in the rest of the discussion as the ``general theorem''. 
The general theorem applies to situations where the intersection curves are transverse and have double points only, 
because the corresponding general construction is understood only when classical Scherk surfaces are used 
to model the desingularizing regions in the vicinity of the intersection curves. 
Therefore we can only consider the cases where the initial configurations in our setting are 
$\Wcal_2$ or $\Wcal_2'$ (recall \ref{E:Wk}) excluding the possibility $k\ge3$. 

Recall that in the first case we have two Clifford tori $\T_1$ and $\T_2$ intersecting orthogonally along 
two totally orthogonal circles $C_1$ and $C_2$. 
In the second case we have three pairwise orthogonal Clifford tori
$\T_1,\T_2,\T'$ with 
six intersection circles $C_1$, $C_2$ , $C'_1$, $C'_2$, 
$C'_3={C'_1}^\perp$, and $C'_4={C'_2}^\perp$, 
where we also have 
$\T_1=\T[C'_2]=\T[C'_4]$, 
$\T_2=\T[C'_3]=\T[C'_1]$, 
$\T'=\T[C_1]=\T[C_2]$, 
$\T_1\cap\T_2=C_1\cup C_2$, 
$\T_1\cap\T'=C'_1\cup C'_3$, 
and 
$\T_2\cap\T'=C'_2\cup C'_4$, 
as follows from \ref{E:Cp}, \ref{E:Cpperp}, and \ref{E:Cpinter} with $k=2$. 
Following the general theorem we define 
$$
\Cunder:=C_1\cup C_2,
\qquad
\Cunder':=C_1\cup C_2 \cup C'_1 \cup C'_2 \cup C'_3 \cup C'_4, 
$$
and 
$\What_2$ or $\What_2'$ (recall \ref{E:Wk}) the abstract surfaces with connected 
components the closures of the connected components of $\Wcal_2\setminus\Cunder$ or 
$\Wcal'_2\setminus\Cunder'$.

Recall now that by the discussion of the Clifford tori in section \ref{initial}, 
any Clifford torus $\T$ is covered isometrically by $\C$ with deck transformations generated by 
$z\to z+\sqrt2 \,\pi$ 
and 
$z\to z+\sqrt2 \,\pi i$.  
The linearized operator for the mean curvature is 
$
\Lcal=\Delta+4
$, 
which clearly has a four dimensional kernel with basis 
$$
\{ \, 
\sin\sqrt2 x \, \sin\sqrt2 y \, , \, 
\sin\sqrt2 x \, \cos\sqrt2 y \, , \, 
\cos\sqrt2 x \, \sin\sqrt2 y \, , \, 
\cos\sqrt2 x \, \cos\sqrt2 y 
\, \},  
$$
where $z=x+iy$ are the standard coordinates on $\C$. 
An alternative basis is given by 
$$
\{ \, 
\sin \sqrt2 (x\pm y) \, , \, 
\cos \sqrt2 (x\pm y) 
\, \}.  
$$

The existence of kernel means that the general theorem cannot be applied unless 
we impose enough symmetry to ensure that the kernel modulo the symmetries becomes trivial. 
To impose these symmetries we consider the scaffolding $\Cmin\subset \Wcal_2$ 
defined by $\Cmin:=C_{0,0} \cup C_{0,\pi/2}$
(recall \ref{E:Cphph})
and the corresponding group $\Grp_{min} := \Grefl(\mathcal{C}_{min}) \subset O(4)$. 
It is easy to calculate then that 
\begin{equation} 
\label{Gmin}
\Grp_{min}= \left\{I_{\Sph^3}, 
      \rot_{C_{0,0}}^\pi, \, 
      \rot_{C_{0,\pi/2}}^\pi, \, 
      \rot_{C_1}^{\pi} \right\}. 
\end{equation} 
Note that for $m\ge1$ we have $\Grp_{min}\subset\Grp_{2,m}$, $\Grp_{min}\subset\Grp'_{2,m}$,  
$\Cmin\subset\mathcal{C}_{2,m}$, and $\Cmin\subset\mathcal{C}'_{2,m}$.  

If we impose $\Grp_{min}$ as the group of symmetries of the construction,
then 
$C_{0,0}\cup C_1\subset\T_1$ has to be contained in the nodal lines of any eigenfunction 
allowed by the symmetries on $\T_1$. 
$\T_1$ this way is subdivided into two flat squares of side length $\pi$. 
The eigenvalues for the Laplacian on each square with Dirichlet boundary data are of the form
$j_1^2+j_2^2$ with $j_1,j_2\in\Z_{>0}$. 
$4$ is not included then. 
Working similarly on $\T_2$
we conclude that there is no kernel modulo the symmetries on $\Wcal_2$. 
We have also to check that there is no kernel on $\What_2$. 
In this case we have to impose an extra Dirichlet condition on $C_2$, 
and then $\T_1$ (or $\T_2$) is subdivided into four flat rectangles of sides 
$\pi$ by $\pi/2$ and the eigenvalues allowed are $j_1^2+ 4j_2^2$
with $j_1,j_2 \in \Z_{>0}$,
and so $4$ is again not included. 

We study now the case of $\Wcal'_2$. 
First we check that there is no kernel on $\T'$. 
Because of the symmetry 
$\rot_{C_1}^{\pi}$ we can assume that we are working on a rectangular (instead of a square) flat torus 
with sides of length $\sqrt2 \pi$ and $\pi/\sqrt2$. 
The eigenvalues of the Laplacian then are $8j_1^2+2j_2^2$ 
with $j_1,j_2 \in \Z_{\ge0}$,
which do not include $4$. 
To check that there is no kernel on $\What'_2$ note first 
that on $\T_1$ Dirichlet conditions are imposed on $C_1$, $C_2$, $C_{0,0}$, $C'_1$, and $C'_3$. 
This subdivides $\T_1$ into eight flat rectangles of sides $\pi/4$ by $\pi$ where the Laplacian 
with Dirichlet conditions on the boundary have eigenvalues 
$16 j_1^2+j_2^2$ with $j_1,j_2\in\Z_{>0}$. 
Similarly for $\T_2$ so it remains only to check $\T'$. 
This has Dirichlet conditions imposed on $C'_1$, $C'_2$, $C'_3$, and  $C'_4$. 
$\T'$ is then subdivided into four flat cylindrical annuli of width $\pi/4$ and so without even 
using the symmetries we have that the smallest eigenvalue is $16$ so that $4$ is again not included. 
Applying then the 
general desingularization theorem announced in \cite[Theorem F]{KapClay} and \cite[Theorem 3.1]{KapSchoen} 
we have the following as a corollary. 

\begin{theorem}
\label{thm-3}
$\Wcal_2$ can be desingularized to produce embedded closed minimal surfaces in $\Sph^3$ symmetric under $\Grp_{min}$
of genus $n_1+n_2+1$, where the towers desingularizing $C_1$ and $C_2$ include $n_1$ and $n_2$ periods respectively, 
provided $n_1$ and $n_2$ are large enough in absolute terms. 
As $n_1,n_2\to\infty$ the minimal surfaces tend to $\Wcal_2$. 

Similarly $\Wcal'_2$ can be desingularized 
to produce embedded closed minimal surfaces in $\Sph^3$ symmetric under $\Grp_{min}$
of genus $n_1+n_2+2n'_1+2n'_2+1$, 
where the towers desingularizing $C_1$ and $C_2$ include $n_1$ and $n_2$ periods respectively, 
the towers desingularizing $C'_1$ and $C'_3$ include $n'_1$ periods,  
and the towers desingularizing $C'_2$ and $C'_4$ include $n'_2$ periods,  
provided $n_1,n_2,n'_1,n'_2$ are large enough in absolute terms. 
As $n_1,n_2,n'_1,n'_2\to\infty$ the minimal surfaces tend to $\Wcal'_2$. 
\end{theorem} 

Note that the main difference of this result compared with the earlier ones is the small symmetry 
imposed and and that the (still large) number of periods along each circle can be prescribed independently on each circle 
(except for the identifications by the symmetries of $C'_1$ with $C'_3$ and $C'_2$ with $C'_4$),  
as opposed to requiring that all numbers have a large common divisor $m$.

\end{document}